\documentclass[a4paper]{amsart}

\usepackage{amsmath,amsthm,amssymb,enumerate}
\usepackage{epsfig}
\usepackage{amssymb}
\usepackage{mathtools}
\usepackage{amsmath}
\usepackage{amssymb}
\usepackage{amsmath,amsthm}
\usepackage[latin1]{inputenc}
\usepackage[T1]{fontenc}
\usepackage{path}
\usepackage{ae,aecompl}
\usepackage{amsfonts}
\usepackage{amsxtra}
\usepackage{euscript,mathrsfs}
\usepackage{color}
\usepackage[left=3.4cm,right=3.4cm,top=4cm,bottom=4cm]{geometry}
\usepackage[citecolor=blue,colorlinks=true]{hyperref}
\allowdisplaybreaks

\usepackage{enumitem}
\setenumerate{label={\rm (\alph{*})}}

\usepackage{amsfonts}
\usepackage{amsxtra}

\usepackage[frak,hyperref,theorem_section]{paper_diening}
\numberwithin{equation}{section}

\newcommand{\Div}{\divergence}

\newcommand{\R}{\mathbb R}
\newcommand{\N}{\mathbb N}

\newcommand{\E}{\mathbb E}
\newcommand{\p}{\mathbb P}

\newcommand{\F}{\mathfrak F}

\newcommand{\A}{\Delta}

\newcommand{\dd}{\mathrm d}
\newcommand{\dx}{\, \mathrm{d}x}

\newcommand{\ds}{\, \mathrm{d}\sigma}
\newcommand{\dt}{\, \mathrm{d}t}
\newcommand{\dxt}{\,\mathrm{d}x\, \mathrm{d}t}
\newcommand{\dxs}{\,\mathrm{d}x\, \mathrm{d}\sigma}
\newcommand{\dif}{\mathrm{d}}

\newcommand{\mf}{\mathfrak{F}}

\newcommand{\prst}{\mathbb{P}}

\newcommand{\mn}{\mathbb{N}}

\newcommand{\mt}{\mathbb{T}^3}
\newcommand{\tor}{\mathbb{T}^3}

\DeclareMathOperator{\diver}{div}

\begin{document}

\title[3D stochastic Navier--Stokes equations]
{Space-time approximation of local strong solutions to the 3D stochastic Navier--Stokes equations}

\author{Dominic Breit}
\address{Institute of Mathematics, TU Clausthal, Erzstra\ss e 1, 38678 Clausthal-Zellerfeld, Germany \\
and Department of Mathematics, Heriot-Watt University, Riccarton Edinburgh EH14 4AS, UK}
\email{dominic.breit@tu-clausthal.de}

\author{Alan Dodgson}
\address{Department of Mathematics, Heriot-Watt University, Riccarton Edinburgh EH14 4AS, UK}
\email{ad335@hw.ac.uk}


%
%

\begin{abstract}
We consider the 3D stochastic Navier--Stokes equation on the torus. Our main result concerns the temporal and spatio-temporal discretisation of a local strong pathwise solution. We prove optimal convergence rates for the energy error with respect to convergence in probability, that is convergence of order 1 in space and of order (up to) 1/2 in time.
The result holds up to the possible blow-up of the (time-discrete) solution. Our approach is based on discrete stopping times for the (time-discrete) solution. 
\end{abstract}


\date{\today}

\maketitle

%
%
%
%
%
%
%
%
%
%

\section{Introduction}

We are concerned with the numerical approximation of the 3D stochastic Navier--Stokes equations which read as
\begin{align}\label{eq:SNS}
\left\{\begin{array}{rc}
\dd\bfu=\mu\Delta\bfu\dt-(\nabla\bfu)\bfu\dt-\nabla p\dt+\Phi(\bfu)\dd W
& \mbox{in $\mathcal Q_T$,}\\
\Div \bfu=0\qquad\qquad\qquad\qquad\qquad\,\,\,\,& \mbox{in $\mathcal Q_T$,}\\
\bfu(0)=\bfu_0\,\qquad\qquad\qquad\qquad\qquad&\mbox{ \,in $\mt$,}\end{array}\right.
\end{align}
$\p$-a.s. in $\mathcal Q_T:=(0,T)\times\mt$, where $T>0$, $\mu>0$ is the viscosity and $\bfu_0$ is a given initial datum. The momentum equation is driven by a cylindrical Wiener process $W$ and the diffusion coefficient $\Phi(\bfu)$ takes values in the space of Hilbert-Schmidt operators; see Section \ref{sec:prob} for details.

The analysis of \eqref{eq:SNS} has a long history starting with the paper \cite{BeTe}, where a semi-deterministic approach is applied.
 A further milestone is the existence of martingale solutions to \eqref{eq:SNS} shown in \cite{FlGa}. These solutions are weak in the analytical sense (derivatives exist only in the sense of distributions and singularities may occur)
 and weak in the probabilistic sense (the probability space is not a priori given but is an integral part of the solution). By now most results from the deterministic case have found their stochastic counterpart, 
 an overview is given in \cite{Fl} and \cite{Ro}. Since the well-posedness of \eqref{eq:SNS} (and its deterministic version) is a big open problem, the existence of weak solutions is the best one can hope unless one is satisfied with a local-in-time result. There exists various results concerning local strong pathwise solutions to \eqref{eq:SNS}, cf. \cite{BeFr,BrzP,GlZi,Kim,Mi}. These solutions are defined on a given stochastic basis and are regular with respect to the spatial variable but only exist up to a stopping time (a precise formulation is given in Definition \ref{def:strsol}). The only information about the latter we have is that it is $\p$-a.s. strictly positive.
 It is yet unclear if the presence of noise changes the well-posedness for \eqref{eq:SNS}. On the one hand, there are results based on the method of convex integration showing that stochastic perturbations do not render the ill-posedness of problems in fluid mechanics, cf. \cite{BFH,HZZ1,HZZ2}.
 On the other hand, it was recently proved in \cite{FL} that a carefully chosen transport noise in \eqref{eq:SNS} can delay the blow-up of the vorticity.

There is less known about the numerical approximation of \eqref{eq:SNS}, though there was recently some progress on the 2D case.
In particular, it is shown in \cite{BrDo} and \cite{CP}
that for any $\xi>0$,
\begin{align}\label{eq:perror}
&\mathbb P\bigg[\max_{1\leq m\leq M}\|\bfu(t_m)-\bfu_{h,m}\|_{L^2_x}^2+\sum_{m=1}^M \tau\|\nabla\bfu(t_m)-\nabla\bfu_{h,m}\|_{L^2_x}^2>\xi\,\big(h^{2\beta}+\tau^{2\alpha}\big)\bigg]\rightarrow0
\end{align}
as $h,\tau\rightarrow0$ (where $\alpha<\frac{1}{2}$ and $\beta<1$ are arbitrary); see also \cite{BeMi1,BeMi2} for related results. Here $\bfu$ is the solution to \eqref{eq:SNS} and $\bfu_{h,m}$ the approximation of $\bfu(t_m)$ with discretisation parameters $\tau=T/M$ (in time) and $h$ (in space). The relation \eqref{eq:perror} tells us that the convergence with respect to convergence in probability is of order (almost) 1/2 in time and 1 in space. A similar convergence result for the pathwise error is not expect due to non-Lipschitz nonlinearity in \eqref{eq:SNS}.
A result such as \eqref{eq:perror} heavily relies on the spatial regularity of the solution and can consequently not be expected for the 3D problem we are interested here. The only reachable
outcome is the approximation of a martingale solution, which converges in law to the solution up to taking a subsequence.
A corresponding result has been proved in \cite{BCP}. 

In this paper we take a different perspective and study the approximation of local strong pathwise solutions. We prove a counterpart of \eqref{eq:perror} which holds locally in time, that is, up to a discrete stopping time which replaces $M$ in maximum and sum.
We obtain a result for the temporal discretisation in Theorem \ref{thm:maint} as well as for the error between time- and space-time discretisation; see Theorem \ref{thm:mainx}.
Both combined give the convergence rate for this spatio-temporal discretisation; see Theorem \ref{thm:maintx}. The analysis of the temporal error in Section \ref{sec:terror} is reminiscent of the estimates for the space-time error for the 2D Dirichlet problem from \cite{BrPr}. They rely on a discrete version of the stopping time for the continuous solution and estimates for the latter. The analysis of the
error between time- and space-time discretisation in Section \ref{sec:txerror} is more delicate and has not been performed in \cite{BrPr}. Building up on ideas from \cite{BrPr2} we introduce a discrete stopping time for the time-discrete solution which announces the blow-up, similar to the stopping time for the continuous solution. In Lemma \ref{lem:sgeqtau},
we prove that for $\tau\rightarrow0$ we can perform a given number of time steps with a high probability before the blow-up. These replaces the strict positivity of the stopping time in the continuous set-up and justifies the subsequent analysis. 

We believe that our approach will be applicable to a wide range of stochastic PDEs which are well-posed locally in time, in particular stochastic Euler equations and stochastic compressible Navier--Stokes equations. Surprisingly, results for the numerical approximation of local solutions to stochastic PDEs do not seem to exist so far.

\section{Mathematical framework}
\label{sec:framework}

\subsection{Probability setup}\label{sec:prob}

Let $(\Omega,\F,(\F_t)_{t\geq0},\prst)$ be a stochastic basis with a complete, right-continuous filtration. The process $W$ is a cylindrical $\mathfrak U$-valued Wiener process, that is, $W(t)=\sum_{j\geq1}\beta_j(t) e_j$ with $(\beta_j)_{j\geq1}$ being mutually independent real-valued standard Wiener processes relative to $(\F_t)_{t\geq0}$, and $(e_j)_{j\geq1}$ a complete orthonormal system in a separable Hilbert space $\mathfrak{U}$.
Let us now give the precise definition of the diffusion coefficient $\varPhi(\bfu)$ taking values in the set of Hilbert-Schmidt operators $L_2(\mathfrak U;\mathbb H)$, where
$\mathbb H$ can take the role of various Hilbert spaces.
We assume that  $\Phi(\bfu)\in L_2(\mathfrak U;L^2(\mt))$ for $\bfu\in L^2(\mt)$, and
$\Phi(\bfu)\in L_2(\mathfrak U;W^{1,2}(\mt))$ for $\bfu\in W^{1,2}(\mt)$, together with
\begin{align}\label{eq:phi0}
\|\Phi(\bfu)-\Phi(\bfv)\|_{L_2(\mathfrak U;L^2_x)}&\leq\,c\|\bfu-\bfv\|_{L^2_x}\qquad\forall \bfu,\bfv\in L^2(\mt),\\
\label{eq:phi1a}
\|\Phi(\bfu)\|_{L_2(\mathfrak U;W^{1,2}_x)}&\leq\,c\big(1+\|\bfu\|_{W^{1,2}_x}\big)\qquad\forall \bfu\in W^{1,2}(\mt),\\
\label{eq:phi1b}
\|D\Phi(\bfu)\|_{L_2(\mathfrak U;\mathcal L( L^{2}(\mt);L^2(\mt)))}&\leq\,c\qquad\forall \bfu\in L^{2}(\mt).
\end{align}
If we are interested in higher regularity some further assumptions are in place and we require additionally 
that $\Phi(\bfu)\in L_2(\mathfrak U;W^{2,2}(\mt))$ for $\bfu\in W^{2,2}(\mt)$, together with
\begin{align}\label{eq:phi2a}
&\|\Phi(\bfu)\|_{L_2(\mathfrak U;W^{2,2}_x)}\leq\,c\big(1+\|\bfu\|_{W^{1,4}_x}^2+\|\bfu\|_{W^{2,2}_x}\big)\qquad\forall \bfu\in W^{2,2}(\mt),\\\label{eq:phi2b}
&\|D^2\Phi(\bfu)\|_{L_2(\mathfrak U;\mathcal L( L^{2}(\mt)\times L^2(\mt);L^2(\mt)))}\leq\,c\qquad\forall \bfu\in L^{2}(\mt).
\end{align}

Assumption \eqref{eq:phi0} allows us to define stochastic integrals.
Given an $(\mathfrak F_t)$-adapted process $\bfu\in L^2(\Omega;C([0,T];L^2(\mt)))$, the stochastic integral $$t\mapsto\int_0^t\varPhi(\bfu)\,\dif W$$
is a well-defined process taking values in $L^2(\mt)$ (see \cite{PrZa} for a detailed construction). Moreover, we can multiply by test functions to obtain
 \begin{align*}
\bigg\langle\int_0^t \varPhi(\bfu)\,\dd W,\bfphi\bigg\rangle_{\!\! L^2_x}=\sum_{j\geq 1} \int_0^t\langle \varPhi(\bfu) e_j,\bfphi\rangle_{L^2_x}\,\dd\beta_j \qquad \forall\, \bfphi\in L^2(\mt).
\end{align*}
Similarly, we can define stochastic integrals with values in $W^{1,2}(\mt)$ and $W^{2,2}(\mt)$ respectively if $\bfu$ belongs to the corresponding class.


\subsection{The concept of solutions}
\label{subsec:solution}

We give the definition of a strong pathwise solution to \eqref{eq:SNS} which exists up to a stopping time $\mathfrak t$. The velocity field belongs $\p$-a.s. to $C([0,\mathfrak t];W^{2,2}_{\diver}(\mt))$.

\begin{definition}\label{def:strsol}
Let $(\Omega,\mf,(\mf_t)_{t\geq0},\prst)$ be a given stochastic basis with a complete right-continuous filtration and an $(\mf_t)$-cylindrical Wiener process $W$. Let $\bfu_0$ be an $\mf_0$-measurable random variable with values in $W^{2,2}_{\diver}(\mt)$. The tuple $(\bfu,\mathfrak t)$ is called a \emph{local strong pathwise solution} \index{incompressible Navier--Stokes system!weak pathwise solution} to \eqref{eq:SNS} with the initial condition $\bfu_0$ provided
\begin{enumerate}
\item $\mathfrak t$ is a $\p$-a.s. strictly positive $(\mathfrak F_t)$-stopping time;
\item the velocity field $\bfu$ is $(\mf_t)$-adapted and
$$\bfu(\cdot\wedge \mathfrak t) \in C([0,T];W^{2,2}_{\diver}(\mt))\cap L^2(0,T;W^{3,2}_{\diver}(\mt)) \quad\text{$\p$-a.s.},$$
\item the momentum equation
\begin{align}\label{eq:mom}
\begin{aligned}
&\int_{\mt}\bfu(t\wedge \mathfrak t)\cdot\bfvarphi\dx-\int_{\mt}\bfu_0\cdot\bfvarphi\dx
\\&=-\int_0^{t\wedge\mathfrak t}\int_{\mt}(\nabla\bfu)\bfu\cdot\bfvarphi\dx\,\dif s+\mu\int_0^{t\wedge\mathfrak t}\int_{\mt}\Delta\bfu\cdot\bfvarphi\dx\,\dif s+\int_{\mt}\int_0^{t\wedge\mathfrak t}\Phi(\bfu) \, \dif W  \cdot\bfvarphi\,\dx
\end{aligned}
\end{align}
holds $\p$-a.s. for all $\bfvarphi\in C^{\infty}_{\diver}(\mt)$ and all $t\geq0$.
\end{enumerate}
\end{definition}
We finally define a maximal strong pathwise solution.
\begin{definition}[Maximal strong pathwise solution]\label{def:maxsol}
Fix a stochastic basis with a cylindrical Wiener process and an initial condition as in Definition \ref{def:strsol}. A triplet $$(\bfu,(\mathfrak{t}_R)_{R\in\N},\mathfrak{t})$$ is a maximal strong pathwise solution to system \eqref{eq:SNS} provided

\begin{enumerate}
\item $\mathfrak{t}$ is a $\p$-a.s. strictly positive $(\mathfrak{F}_t)$-stopping time;
\item $(\mathfrak{t}_R)_{R\in\mn}$ is an increasing sequence of $(\mathfrak{F}_t)$-stopping times such that
$\mathfrak{t}_R<\mathfrak{t}$ on the set $[\mathfrak{t}<\infty]$,
$\lim_{R\to\infty}\mathfrak{t}_R=\mathfrak t$ $\p$-a.s., and
\begin{equation}\label{eq:blowup}
\mathfrak t_R:=\inf \big\{t\in[0,\infty):\,\,\|\bfu(t)\|_{W^{2,2}_x}\geq R\big\}\quad \text{on}\quad [\mathfrak{t}<\infty] ,
\end{equation}
with the convention that $\mathfrak{t}_R=\infty$ if the set above is empty;
\item each triplet $(\bfu,\mathfrak{t}_R)$, $R\in\mn$,  is a local strong pathwise solution in the sense  of Definition \ref{def:strsol}.
\end{enumerate}
\end{definition}
The following result can be proved along the lines of \cite{Kim}, where the stochastic Navier--Stokes equations on the whole space $\R^3$ are considered with fractional differentiability $\sigma\in(3/2,2)$. As mentioned on \cite[page 2]{Kim} the case of differentiability $\sigma =2$ we are interested in is even easier.
\begin{theorem}\label{thm:inc2d}
Suppose that \eqref{eq:phi0}--\eqref{eq:phi2b} hold, and
 that $\bfu_0\in L^2(\Omega,W^{1,2}_{\diver}(\mt))$. Then there is a
unique maximal global strong pathwise solution to \eqref{eq:SNS} in the sense  of Definition \ref{def:maxsol}.
\end{theorem}

\subsection{Finite elements}
We work with a standard finite element set-up for incompressible fluid mechanics, see e.g. \cite{GR}.
We denote by $\mathscr{T}_h$ a quasi-uniform subdivision of $\mt$ into simplices of maximal diameter $h>0$.   
For $K\subset \setR ^3$ and $\ell\in \setN _0$ we denote by
$\mathscr{P}_\ell(K)$ the polynomials on $K$ of degree less than or equal
to $\ell$. 
Let us
characterize the finite element spaces $V^h(\mt)$ and $P^h(\mt)$ as
\begin{align*}
  V^h(\mt)&:= \set{\bfv_h \in W^{1,2}(\mt)\,:\, \bfv_h|_{K}
    \in \mathscr{P}_i(K)\,\,\forall K\in \mathscr T_h},\\
P^h(\mt)&:=\set{\pi_h \in L^{2}(\mt)\,:\, \pi_h|_{K}
    \in \mathscr{P}_j(K)\,\,\forall K\in \mathscr T_h}.
\end{align*}
We will assume that $i,j\in\N$ to get \eqref{eq:stab'} below.
In order to guarantee stability of our approximations we relate $V^h(\mt)$ and $P^h(\mt)$ by the discrete inf-sup condition, that is we assume that
\begin{align*}
\sup_{\bfv_h\in V^h(\mt)} \frac{\int_{\mt}\Div\bfv_h\,\pi_h\dx}{\|\nabla\bfv_h\|_{L^2_x}}\geq\,C\,\|\pi_h\|_{L^2_x}\quad\,\forall\pi_h\in P^h(\mt),
\end{align*}
where $C>0$ does not depend on $h$. This gives a relation between $i$ and $j$ (for instance the choice $(i,j)=(1,0)$ is excluded whereas $(i,j)=(2,0)$ is allowed).
Finally, we define the space of discretely solenoidal finite element functions by 
\begin{align*}
  V^h_{\Div}(\mt)&:= \bigg\{\bfv_h\in V^h(\mt):\,\,\int_{\mt}\Div\bfv_h\,\,\pi_h\dx=0\,\,\forall\pi_h\in P^h(\mt)\bigg\}.
\end{align*}
Let $\Pi_h:L^2(\mt)\rightarrow V_{\Div}^h(\mt)$ be the $L^2(\mt)$-orthogonal projection onto $V_{\Div}^h(\mt)$. The following results concerning the approximability of $\Pi_h$ are well-known (see, for instance \cite{HR}). There is $c>0$ independent of $h$ such that we have
  \begin{align}
    \label{eq:stab}
 \int_{\mt} \Big|\frac{\bfv-\Pi_h \bfv}{h}\Big|^2\dx+ \int_{\mt} \abs{\nabla\bfv-\nabla\Pi_h \bfv}^2\dx &\leq
    \,c\, \int_{\mt} \abs{\nabla
      \bfv}^2\dx
  \end{align}
for all $\bfv\in W^{1,2}_{\Div}(\mt)$, and
  \begin{align}
    \label{eq:stab'}
 \int_{\mt} \Big|\frac{\bfv-\Pi_h \bfv}{h}\Big|^2\dx+ \int_{\mt} \abs{\nabla\bfv-\nabla\Pi_h \bfv}^2\dx &\leq
    \,c\,h^2 \int_{\mt} \abs{\nabla^2
      \bfv}^2\dx
  \end{align}
for all $\bfv\in W^{2,2}_{\Div}(\mt)$. Similarly, if $\Pi_h^\pi:L^2(\mt)\rightarrow P^h(\mt)$ denotes the $L^2(\mt)$-orthogonal projection onto $P^h(\mt)$, we have
\begin{align}
\label{eq:stabpi}
 \int_{\mt} \Big|\frac{p-\Pi_h^\pi p}{h}\Big|^2\dx &\leq
    \,c\, \int_{\mt} \abs{\nabla
      p}^2\dx
\end{align}
for all $p\in W^{1,2}(\mt)$, and
\begin{align}
\label{eq:stabpi'}
 \int_{\mt} \Big|\frac{p-\Pi_h^\pi p}{h}\Big|^2\dx &\leq
    \,c\,h^2 \int_{\mt} \abs{\nabla^2
      p}^2\dx
\end{align}
for all $p\in W^{2,2}(\mt)$. Note that \eqref{eq:stabpi'} requires the assumption $j\geq1$ in the definition of $P^h(\mt)$, whereas \eqref{eq:stabpi} also holds for $j=0$.

\section{Regularity of solutions}
\subsection{Estimates for the continuous solution}
In this section we derive the crucial estimates for the
continuous solution, which hold up to the stopping time $\mathfrak t_R$ for some $R\gg1$.
\begin{lemma}\label{lem:reg}
 \begin{enumerate}
Let $(\Omega,\mf,(\mf_t)_{t\geq0},\prst)$ be a given stochastic basis with a complete right-continuous filtration and an $(\mf_t)$-cylindrical Wiener process $W$. Suppose that $\bfu_0$ is an $\mathfrak F_0$-measurable random variable with values in $W^{2,2}_{ \Div}(\mt)$.
Let $(\bfu,(\mathfrak t_R)_{R\in\N},\mathfrak t)$ be the maximal strong pathwise solution to \eqref{eq:SNS}, cf. Definition \ref{def:maxsol}.
\item[(a)] Assume 
that $\bfu_0\in L^r(\Omega,L^{2}_{\Div}(\mt))$ for some $r\geq2$ and 
that $\Phi$ satisfies \eqref{eq:phi0}. Then we have
\begin{align}\label{eq:W12}
\E\bigg[ \bigg(\sup_{0\leq t\leq T}\int_{\mt}|\bfu(t\wedge\mathfrak t_R)|^2\dx+\int_0^{T\wedge t_R}\int_{\mt}|\nabla\bfu|^2\dxt\bigg)^{\frac{r}{2}}\bigg]\leq\,c\,\E\Big[1+\|\bfu_0\|_{L^2_x}^r\Big].
\end{align}
\item[(b)] Assume 
that $\bfu_0\in L^r(\Omega,W^{1,2}_{0,\Div}(\mt))$ for some $r\geq2$ and 
that $\Phi$ satisfies \eqref{eq:phi0}--\eqref{eq:phi1b}. Then we have
\begin{align}\label{eq:W22}\begin{aligned}
\E\bigg[\bigg(\sup_{0\leq t\leq T}\int_{\mt}|\nabla\bfu(t\wedge \mathfrak t_R)|^2\dx&+\int_0^{T\wedge \mathfrak t_R}\int_{\mt}|\nabla^2\bfu|^2\dxt\bigg)^{\frac{r}{2}}\bigg]\\&\leq\,cR^{3r}\,\E\Big[1+\|\bfu_0\|_{W^{1,2}_x}^r\Big].
\end{aligned}
\end{align}
\item[(c)] Assume 
that $\bfu_0\in L^r(\Omega,W^{2,2}(\mt))$ for some $r\geq2$ and 
that assumptions \eqref{eq:phi0}--\eqref{eq:phi2b} hold. Then we have
\begin{align}\label{eq:W32}
\begin{aligned}
\E\bigg[\bigg(\sup_{0\leq t\leq T}\int_{\mt}|\nabla^2\bfu(t\wedge \mathfrak t_R)|^2\dx&+\int_0^{T\wedge \mathfrak t_R}\int_{\mt}|\nabla^3\bfu|^2\dxt\bigg)^{\frac{r}{2}}\bigg]\\&\leq\,cR^{3r}\,\E\Big[1+\|\bfu_0\|_{W^{2,2}_x}^r\Big].
\end{aligned}
\end{align}
\end{enumerate}
Here $c=c(r,T)$ is independent of $R$.
\end{lemma}
\begin{proof}
Let us suppose that $\bfu_0\in L^\infty(\Omega,W^{2,2}(\mt))$. This assumption can be removed eventually by truncating $\bfu_0$.
Similar to \cite{Mi} we consider the solution to a truncated problem. For $R>1$ and $\zeta\in C_c^\infty([0,2))$ with $0\leq \zeta\leq 1$ and $\zeta=1$ in $[0,1]$ we set $\zeta_R:=\zeta(R^{-1}\cdot)$ . 
Let $\bfu^R$ be an $(\mathfrak F_t)$-adapted process with
$$\bfu^R \in C([0,T];W^{2,2}_{\diver}(\tor))\cap L^2(0,T; W^{3,2}_{\diver}(\tor))\quad\text{$\p$-a.s.}$$
 such that
\begin{align}\label{eq:trun}
\begin{aligned}
\int_{\mt}\bfu^R(t)\cdot\bfvarphi\dx&=\int_{\mt}\bfu_0\cdot\bfvarphi\dx
+\int_0^t\int_{\mt}\zeta_R(\|\bfu^{R}\|_{W^{2,2}_x})\bfu^{R}\otimes\bfu^{R} :\nabla\bfphi\dx\,\dif s\\&-\mu\int_0^t\int_{\mt}\nabla\bfu^R:\nabla\bfvarphi\dx\,\dif s+\int_0^t\zeta_R(\|\bfu^{R}\|_{W^{2,2}_x})\int_{\mt}\Phi(\bfu^R)\cdot\bfvarphi\dx\,\dif W,
\end{aligned}
\end{align}
holds $\p$-a.s. for all $\bfvarphi\in W^{1,2}_{\diver}(\mt)$ and all $t\in[0,T]$.
It can be shown by means of a Glarking approximation
(and highger order energy estimates which hold thanks to the periodic boundary conditions)
that a unique strong pathwise solution to \eqref{eq:trun} exists.
We finally note
that $\bfu^R(\cdot\wedge\mathfrak t_R)=\bfu(\cdot\wedge \mathfrak t_R)$ such that is sufficient to prove the claimed estimates for $\bfu^R$ instead of $\bfu$.

Estimate \eqref{eq:W12} is the standard a priori estimate which can be proved by applying It\^{o}'s formula to the functional $t\mapsto \|\bfu^R\|_{L^2_x}^2$ and using the cancellation of the convective term.

As far as \eqref{eq:W22} is concerned we can apply
It\^{o}'s formula to $t\mapsto\|\nabla\bfu^R\|_{L^2_x}^2$ and use \eqref{eq:trun}, which yields
\begin{align*}
\int_{\mt}|\nabla\bfu^{R}|^2\dx&=\int_{\mt}|\nabla\bfu_0|^2\dx
+2\int_0^t\int_{\mt}\zeta_R(\|\bfu^{R}\|_{W^{2,2}_x})(\bfu^{R}\cdot\nabla)\bfu^{R}\cdot\A\bfu^{R}\dx\,\dif s\\&-2\mu\int_0^t\int_{\mt}|\A\bfu^{R}|^2\dx\,\dif s+2\int_0^t\int_{\mt}\zeta_R(\|\bfu^{R}\|_{W^{2,2}_x})\Phi(\bfu^{R})\cdot\A\bfu^{R}\dx\,\dif W\\
&+\sum_{k=1}^\infty\int_0^t\bigg(\zeta_R(\|\bfu^{R}\|_{W^{2,2}_x})\int_{\mt}\nabla\{\Phi(\bfu^{R})e_k\}\dx\bigg)^2\,\dif s\\
&=:\mathrm{I}(t)+\dots+\mathrm{V}(t).
\end{align*}
We only show how to estimate the
convective term $\mathrm{II}$, which significantly differs from the 2D case. We refer to \cite[Lemma 3.1]{BrPr}, where it is shown how to control the remaining integrals independently of $R$.
We have by definition of $\zeta_R$ and the embedding $W^{2,2}(\mt)\hookrightarrow L^\infty(\mt)$
\begin{align*}
\mathrm{II}(t)
&\leq 2\int_0^t\zeta_R(\|\bfu^{R}\|_{W^{2,2}_x})\|\bfu^{R}\|_{L^\infty_x}\|\nabla\bfu^{R}\|_{L^2_x}\|\A\bfu^{N}\|_{L^2_x}\dif s\leq \,cR^3.
\end{align*}
For (c) is we argue similarly to (b)\footnote{In \cite[Lemma 3.1. (c)]{BrPr} a completely different approach is used for the corresponding estimate due to
problems related to Dirichlet boundary conditions.} and apply
It\^{o}'s formula to the mapping $t\mapsto \int_{\mt}|\A \bfu^{R}(t)|^2\dx$ which shows
\begin{align*}
\int_{\mt}|\A\bfu^{R}|^2\dx&=\int_{\mt}|\A\bfu_0|^2\dx
+2\int_0^t\int_{\mt}\zeta_R(\|\bfu^{R}\|_{W^{2,2}_x})(\bfu^{R}\cdot\nabla)\bfu^{R}\cdot\A^2\bfu^{R}\dx\,\dif s\\&-2\mu\int_0^t\int_{\mt}|\nabla\A\bfu^{R}|^2\dx\,\dif s+2\int_0^t\zeta_R(\|\bfu^{R}\|_{W^{2,2}_x})\int_{\mt}\Phi(\bfu^{R})\cdot\A^2\bfu^{R}\dx\,\dif W\\
&+\sum_{k=1}^\infty\int_0^t\bigg(\zeta_R(\|\bfu^{R}\|_{W^{2,2}_x})\int_{\mt}\nabla^2\{\Phi(\bfu^{R})e_k\}\dx\bigg)^2\,\dif s\\
&=:\mathrm{VI}(t)+\dots+\mathrm{X}(t).
\end{align*}
It holds using the embedding $W^{2,2}(\mt)\hookrightarrow W^{1,4}(\mt)$ that 
\begin{align*}
\mathrm{VII}(t)&\leq 2\int_0^t\zeta_R(\|\nabla\bfu^{R}\|_{L^2_x})\|\bfu^{R}\|_{L^\infty_x}\|\nabla^2\bfu^{R}\|_{L^2_x}\|\nabla\A\bfu^{R}\|_{L^2_x}\dif s\\
&+2\int_0^t\zeta_R(\|\nabla\bfu^{R}\|_{L^2_x})\|\nabla\bfu^{R}\|_{L^4_x}^2\|\nabla\A\bfu^{R}\|_{L^2_x}\dif s\\
&\leq \,cR\int_0^t\|\nabla^2\bfu^{R}\|_{L^2_x}\|\nabla\A\bfu^{R}\|_{L^2_x}\dif s\\
&\leq \,c(\delta)R^2\int_0^t\|\nabla^2\bfu^{R}\|_{L^2_x}^2\dif s+\,\delta\int_0^t\|\nabla\A\bfu^{R}\|_{L^2_x}^2\dif s
\end{align*}
for any $\delta>0$. The expectation of the first term can be controlled by the estimates from (b).
As for the stochastic terms, we obtain
\begin{align*}
\mathrm{X}(t)&\leq\sum_{k\geq1}\int_0^{t}\int_{\mt}|\zeta_R(\|\bfu^{R}\|_{W^{2,2}_x})\A\Phi(\bfu^{R})e_k|^2\dx\,\dif s\\
&\leq\,c\int_0^{t}\zeta_R(\|\bfu^{R}\|_{W^{2,2}_x})\big(1+\|\bfu^{R}\|_{W^{1,4}_x}^4+\|\bfu^{R}\|^2_{W^{2,2}_x}\big)\,\dif s\\
&\leq\,cR^2\int_0^t\big(1+\|\nabla^2\bfu^{R}\|_{L^2_x}^2\big)\dif s
\end{align*}
as well as
\begin{align*}
\E\bigg[\sup_{0\leq t\leq T}|\mathrm{IX}(t)|^{\frac{r}{2}}\bigg]
&\leq c\,\E\bigg[\bigg(\int_0^{T}
\zeta_R(\|\bfu^{R}\|_{W^{2,2}_x})\big(1+\|\bfu^{R}\|_{W^{1,4}_x}^4+\|\bfu^{R}\|^2_{W^{2,2}_x}\big)\|\Delta\bfu^{R}\|^2_{L^2_x}\dt\bigg)^{\frac r4}\bigg]\\
&\leq  \,cR^{\frac{3r}{2}}.
\end{align*}
The proof is complete.
\end{proof}

\subsection{Stochastic pressure decomposition}\label{sec:stochpress}
Since we will be working with discretely divergence-free function spaces in the finite-element analysis, it is inevitable to introduce the pressure function. 
For $\bfphi\in C^\infty(\mt)$ we can insert $$\mathcal P\bfphi:=\bfphi-\nabla\Delta^{-1}\Div\bfphi$$
in \eqref{eq:mom}. Here $\Delta^{-1}$ is the solution operator to the Laplace equation with respect to periodic boundary conditions on $\mt$. Note that $\Delta^{-1}$ satisfies
\begin{align}\label{eq:Delta1}
\Delta^{-1}&:W^{-2,p}(\mt)\rightarrow  L^{p}(\mt),\\
\label{eq:Delta2}
\Delta^{-1}&:W^{-1,p}(\mt)\rightarrow W^{1,p}(\mt),\\
\label{eq:Delta3}
\Delta^{-1}&:W^{r,p}(\mt)\rightarrow W^{r+2,p}(\mt),
\end{align}
for all $p\in(1,\infty)$ and all $r\in\N$, where $W^{-k,p}(\mt)=\big(W^{k,p'}(\mt)\big)'$ for $k\in\N$.
We obtain
\begin{align}
\nonumber
\int_{\mt}\bfu(t\wedge \mathfrak t_R)\cdot\bfvarphi\dx &-\int_0^{t\wedge \mathfrak t_R}\int_{\mt}\mu\Delta \bfu\cdot\bfphi\dxs+\int_0^{t\wedge \mathfrak t_R}\int_{\mt}(\nabla\bfu)\bfu\cdot\bfphi\dxs\\
\label{eq:pressure}&=\int_{\mt}\bfu(0)\cdot\bfvarphi\dx
+\int_0^{t\wedge \mathfrak t_R}\int_{\mt}\pi_{\mathrm{det}}\,\Div\bfphi\dxs
\\
\nonumber&+\int_{\mt}\int_0^{t\wedge \mathfrak t_R}\Phi(\bfu)\,\dd W\cdot \bfvarphi\dx+\int_{\mt}\int_0^{t\wedge \mathfrak t_R}\Phi^\pi\,\dd W\cdot \bfvarphi\dx,
\end{align}
where 
\begin{align*}
\pi_{\mathrm{det}}&=-\Delta^{-1}\Div\big((\nabla\bfu)\bfu\big),\\
\Phi^\pi&=-\nabla\Delta^{-1}\Div\Phi(\bfu).
\end{align*}
This corresponds to the stochastic pressure decomposition introduced in \cite{Br} (see also \cite[Chap. 3]{Br2} for a slightly different presentation). 
In the following we will analyse how the regularity of $\bfu$ transfers to $\pi_{\mathrm{det}}$
and $\Phi^\pi$. Arguing as in \cite[Corollary 2.5]{BrDo} and using \eqref{eq:Delta1}--\eqref{eq:Delta3} we obtain
\begin{align*}
\E\bigg[\bigg(\int_0^{T\wedge \mathfrak t_R}\|\pi_{\mathrm det}\|_{W^{\ell,2}_x}^2\dt\bigg)^{\frac{r}{4}}\bigg]
&\leq\,c\,\E\bigg[\bigg(\sup_{0\leq t\leq T}\|\bfu(t\wedge \mathfrak t_R)\|_{W^{\ell,2}_x}^2+\int_0^{T\wedge \mathfrak t_R}\|\bfu\|_{W^{\ell+1,2}_x}^2\dt\bigg)^{\frac{r}{2}}\bigg]
\end{align*}
as well as
\begin{align*}
\E\bigg[\bigg(\sup_{t\in[0, T\wedge \mathfrak t]}\|\Phi^\pi\|_{L_2(\mathfrak U;W^{\ell,2}_x)}^2\bigg)^{\frac{r}{2}}\bigg]
&\leq\,c\,\E\bigg[1+\sup_{0\leq t\leq T}\|\bfu(t\wedge \mathfrak t_R)\|_{W^{\ell,2}_x}^r\bigg]
\end{align*}
for $\ell\in\{0,1,2\}$ (note that $W^{0,p}(\mt)=L^p(\mt)$ for $p\in[1,\infty]$). Consequently, Lemma \ref{lem:reg} implies the following.
\begin{lemma}\label{lem:pressure}
\begin{enumerate}
\item Under the assumptions of Lemma \ref{lem:reg} (a) we have
\begin{align*}
\E\bigg[\bigg(\int_0^{T\wedge \mathfrak t_R}\|\pi_{\mathrm det}\|_{L^{2}_x}^2\dt+\sup_{0\leq t\leq T}\|\Phi^\pi(t\wedge \mathfrak t_R)\|_{L_2(\mathfrak U;L^{2}_x)}^2\bigg)^{\frac{r}{2}}\bigg]\leq\,c\,\E\Big[1+\|\bfu_0\|_{L^2_x}^r\Big].
\end{align*}
\item Under the assumptions of Lemma \ref{lem:reg} (b) we have
\begin{align*}
\E\bigg[\bigg(\int_0^{T\wedge \mathfrak t_R}\|\pi_{\mathrm det}\|_{W^{1,2}_x}^2\dt+\sup_{0\leq t\leq T}\|\Phi^\pi(t\wedge \mathfrak t_R)\|_{L_2(\mathfrak U;W^{1,2}_x)}^2\bigg)^{\frac{r}{2}}\bigg]\leq\,cR^{3r}\,\E\Big[1+\|\bfu_0\|_{W^{1,2}_x}^r\Big].
\end{align*}
\item Under the assumptions of Lemma \ref{lem:reg} (c) we have
\begin{align*}
\E\bigg[\bigg(\int_0^{T\wedge \mathfrak t_R}\|\pi_{\mathrm det}\|_{W^{2,2}_x}^2\dt+\sup_{0\leq t\leq T}\|\Phi^\pi(t\wedge \mathfrak t_R)\|_{L_2(\mathfrak U;W^{2,2}_x)}^2\bigg)^{\frac{r}{2}}\bigg]\leq\,cR^{3r}\,\E\Big[1+\|\bfu_0\|_{W^{2,2}_x}^r\Big].
\end{align*}
\end{enumerate}
Here $c=c(r,T)$ is independent of $R$.
\end{lemma}

As in \cite[Corollary 2.6]{BrDo}, we can combine Lemmas \ref{lem:reg} and \ref{lem:pressure} to conclude the following result concerning the time regularity of $\bfu$.

\begin{corollary}\label{cor:uholder}
\begin{enumerate}
\item Let the assumptions of Lemma \ref{lem:reg} (b) be satisfied for some $r>2$. Then we have
\begin{align}
\label{eq:holder}
\E\Big[\Big(\|\bfu(\mathfrak t_R\wedge\cdot)\|_{C^\alpha([0,T];L^{2}_x)}\Big)^{\frac{r}{2}}\Big]\leq cR^{3r}\,\E\Big[1+\|\bfu_0\|_{W^{1,2}_x}^r\Big]
\end{align}
for all $\alpha<\frac{1}{2}$.
\item Let the assumptions of Lemma \ref{lem:reg} (c) be satisfied for some $r>2$. Then we have
\begin{align}
\label{eq:holder}
\E\Big[\Big(\|\bfu(\mathfrak t_R\wedge\cdot)\|_{C^\alpha([0,T];W^{1,2}_x)}\Big)^{\frac{r}{2}}\Big]\leq cR^{3r}\,\E\Big[1+\|\bfu_0\|_{W^{2,2}_x}^r\Big]
\end{align}
for all $\alpha<\frac{1}{2}$.
\end{enumerate}
Here $c=c(r,T,\alpha)$ is independent of $R$.
\end{corollary}

\section{Time discretisation}\label{sec:time}
We now consider a temporal approximation of  \eqref{eq:SNS} on an equidistant partition of $[0,T]$ with mesh size $\tau=T/M$, and set $t_m=m\tau$. 
Let $\bfu_{0}$ be an $\mathfrak F_0$-measurable random variable with values in $L^{2}_{\Div}(\mt)$. For $1 \leq m \leq M$, we aim at constructing iteratively a sequence of $\mathfrak F_{t_m}$-measurable random variables $\bfu_{m}$ with values in $W^{1,2}_{\Div}(\mt)$ such that
for every $\bfphi\in W^{1,2}_{\Div}(\mt)$ it holds true $\p$-a.s.
\begin{align}\label{tdiscr}
\begin{aligned}
\int_{\mt}&\bfu_{m}\cdot\bfvarphi \dx -\tau\int_{\mt}\left( \bfu_{m}\otimes\bfu_{m-1} \right) :\nabla\bfphi\dx\\
&\qquad=-\mu\tau\int_{\mt}\nabla\bfu_{m}:\nabla\bfphi\dx+\int_{\mt}\bfu_{m-1}\cdot\bfvarphi \dx+\int_{\mt}\Phi(\bfu_{m-1})\,\Delta_mW\cdot \bfvarphi\dx,
\end{aligned}
\end{align}
where $\Delta_m W=W(t_m)-W(t_{m-1})$. 
For given $\bfu_{m-1}$ and $\Delta_m W$, verifying the existence of a unique $\bfu_m$ solving \eqref{tdiscr} is straightforward since the problem is linear in $\bfu_m$.

%
%
\begin{lemma}\label{lemma:3.1a} 
Assume that $\bfu_0\in L^{2^q}(\Omega,L^{2}_{\Div}(\mt;\R^2))$ for some $q\in\N$ and that \eqref{eq:phi0} holds. Then the iterates $(\bfu_m)_{m=1}^M$ given by \eqref{tdiscr}
satisfy the following estimate uniformly in $M$:
\begin{align}
\label{lem:3.1a}\E\bigg[\max_{1\leq m\leq M}\|\bfu_m\|^{2^q}_{L^{2}_x}+\tau\sum_{m=1}^M\|\bfu_m\|_{L^{2}_x}^{2^q-2}\|\nabla\bfu_m\|^2_{L^2_x}\bigg]&\leq\,c,
\end{align}
where $c=c(q,T,\Phi,\bfu_0)>0$. 
\end{lemma}
\begin{proof}
The proof of \eqref{lem:3.1a} is identical to \cite[Lemma 3.1]{BCP}. Note that, different from \cite{BCP}, we consider a semi-implicit algorithm, which does not impact the proof since the convective term still cancels when testing with $\bfu_m$.
 \end{proof}
 
\subsection{Estimates for the time-discrete solution}
\label{sec:disest}

 In order to obtain ``local-in-time'' estimates we consider
 a truncated variant of \eqref{tdiscr}, similarly to \eqref{eq:trun}. Let $\bfu_m^R\in W^{1,2}_{\Div}(\mt)$ be such that
 for every $\bfphi\in W^{1,2}_{\Div}(\mt)$, it holds true $\p$-a.s. that
\begin{align}\label{tdiscrR}
\begin{aligned}
\int_{\mt}\bfu_{m}^R\cdot\bfvarphi \dx &-\tau\int_{\mt}\zeta_R(\|\bfu_{m-1}^R\|_{W^{2,2}_x})\bfu^R_{m}\otimes\bfu_{m-1}^R :\nabla\bfphi\dx\\
&=-\mu\tau\int_{\mt}\nabla\bfu_{m}^R:\nabla\bfphi\dx+\int_{\mt}\bfu_{m-1}^R\cdot\bfvarphi \dx\\&+\int_{\mt}\zeta_R(\|\bfu_{m-1}^R\|_{W^{2,2}_x})\Phi(\bfu_{m-1}^R)\,\Delta_mW\cdot \bfvarphi\dx
\end{aligned}
\end{align}
with $\bfu_0^R=\bfu_0$. Arguing as for Lemma \ref{lemma:3.1a} 
and noticing that the convective term in \eqref{tdiscrR} still cancels when testing with $\bfu_m^R$ once can show for $q\in\N$
\begin{align}
\label{lem:3.1aR}\E\bigg[\max_{1\leq m\leq M}\|\bfu_m^R\|^{2^q}_{L^{2}_x}+\tau\sum_{m=1}^M\|\bfu_m^R\|_{L^{2}_x}^{2^q-2}\|\nabla\bfu^R_m\|^2_{L^2_x}\bigg]&\leq\,c
\end{align}
where $c=c(q,T,\Phi,\bfu_0)>0$ is independent of $R$. 
We are now going to prove $R$-dependent estimates for $(\bfu_m^R)_{m=1}^M$ which we transfer eventually to $(\bfu_m)_{m=1}^M$ by introducing a suitable discrete stopping time.


\begin{lemma}\label{lemma:3.1} 
Assume that $\bfu_0\in L^{2^q}(\Omega,W^{1,2}_{\Div}(\mt))$ for some $q\in\N$ and suppose that \eqref{eq:phi0}--\eqref{eq:phi1b} hold. Then the iterates $(\bfu_m^R)_{m=1}^M$ given by \eqref{tdiscrR}
satisfy the following estimates uniformly in $M$:
\begin{align}
\label{lem:3.1b}
\begin{aligned}
\E\bigg[\max_{1\leq m\leq M}\|\bfu_m^R\|^{2^q}_{W^{1,2}_x}&+\sum_{m=1}^{M}\tau\|\bfu_m^R\|_{W^{1,2}_x}^{2^q-2}\|\nabla^2\bfu_m^R\|^2_{L^2_x}\bigg]\\&+\E\bigg[\sum_{m=1}^{M}\|\bfu_m^R\|_{W^{1,2}_x}^{2^q-2}\|\nabla(\bfu_m^R-\bfu_{m-1}^R)\|^2_{L^2_x}\bigg]\leq\,ce^{cR^{2}},
\end{aligned}
\end{align}
where $c=c(q,T,\Phi,\bfu_0)>0$ is independent of $R$.
\end{lemma}
\begin{proof}
The proof is reminiscent of \cite[Lemma 2.8]{BrPr2} but differs in the estimates for the convective term and the nonlinear diffusion coefficient.
We proceed formally; a rigorous proof can be obtained using a Galerkin approximation.
We test \eqref{tdiscrR} by $\Delta\bfu_m$, obtaining
 \begin{align*}
\int_{\mt}\nabla(\bfu_{m}^R-&\bfu_{m-1}^R):\nabla\bfu_m^R\dx +\mu\tau\int_{\mt}|\Delta\bfu_m^R|^2\dx\\&=-\tau\int_{\mt}\zeta_R(\|\bfu_{m-1}^R\|_{W^{2,2}_x})(\nabla\bfu_m^R)\bfu_{m-1}^R\cdot\Delta\bfu^R_m\dx\\&+\zeta_R(\|\bfu_{m-1}^R\|_{W^{2,2}_x})\int_{\mt}\Delta\bfu_m^R\cdot\Phi(\bfu_{m-1}^R)\Delta_mW\dx.
\end{align*}
For $\delta>0$ we have by the definition of $\zeta_R$ and the embedding $W^{2,2}(\mt)\hookrightarrow L^\infty(\mt)$
\begin{align}\label{eq:1407}
\begin{aligned}
\int_{\mt}\zeta_R&(\|\bfu_{m-1}^R\|_{W^{2,2}_x})(\nabla\bfu_m^R)\bfu_{m-1}^R\cdot\Delta\bfu^R_m\dx \\&\leq\int_{\mt}\zeta_R(\|\bfu_{m-1}^R\|_{W^{2,2}_x})\|\bfu_{m-1}^R\|_{L^\infty_x}\|\nabla\bfu_m^R\|_{L^2_x}\|\Delta\bfu_m^R\|_{L^2_x}
\\&\leq\,cR\|\nabla\bfu_m\|_{L^2_x}\|\Delta\bfu_m\|_{L^2_x}
\\&\leq\,c(\delta)R^2\|\nabla\bfu_m\|_{L^2_x}^2+\delta\|\Delta\bfu_m\|_{L^2_x}^{2}.
\end{aligned}
\end{align}
Summing up and choosing $\delta$ sufficiently small shows that
 \begin{align*}
\tfrac{1}{2}\int_{\mt}|\nabla\bfu_{m}^R|^2\dx&+\tfrac{1}{2}\sum_{n=1}^m\int_{\mt}|\nabla(\bfu_{n}^R-\bfu^R_{n-1})|^2\dx +\tfrac{\mu}{2} \sum_{n=1}^m\tau\int_{\mt}|\A\bfu^R_{n}|^2\dx\\&\leq \tfrac{1}{2} \int_{\mt}|\nabla\bfu_{0}|^2\dx+cR^2 \sum_{n=1}^m\tau\|\nabla\bfu_n^R\|_{L^2_x}^2+\mathscr M^1_m+\mathscr M_m^2,
\end{align*}
where
\begin{align*}
\mathscr M^1_m&=\sum_{n=1}^m\zeta_R(\|\bfu_{n-1}^R\|_{W^{2,2}_x})\int_{\mt}\int_{t_{n-1}}^{t_n}\Phi(\bfu_{m-1}^R)\,\dd W\cdot \A\bfu^R_{n-1}\dx,\\
\mathscr M_m^2&=\sum_{n=1}^m\zeta_R(\|\bfu_{n-1}^R\|_{W^{2,2}_x})\int_{\mt}\int_{t_{n-1}}^{t_n}\Phi(\bfu_{m-1}^R)\,\dd W\cdot \A(\bfu^R_n-\bfu^R_{n-1})\dx.
\end{align*}
By the discrete Gronwall lemma we have $\p$-a.s.
 \begin{align*}
\tfrac{1}{2}\max_{1\leq m\leq M}\int_{\mt}|\nabla\bfu^R_{m}|^2\dx&+\tfrac{1}{2}\sum_{n=1}^{M}\int_{\mt}|\nabla(\bfu^R_{n}-\bfu^R_{n-1})|^2\dx +\tfrac{\mu}{2} \sum_{n=1}^{M}\tau\int_{\mt}|\A\bfu^R_{n}|^2\dx\\&\leq ce^{cR^2}\bigg( \int_{\mt}|\nabla\bfu_{0}|^2\dx+\max_{1\leq m\leq M}|\mathscr M^1_m|+\max_{1\leq m\leq M}|\mathscr M_m^2|\bigg).
\end{align*}
Since $\bfu_{m-1}$ is $\mathfrak F_{t_{m-1}}$-measurable we know that $\mathscr M^1_m$ is an $(\mathfrak F_{t_m})$-martingale. Consequently, by the Burkholder-Davis-Gundy inequality, \eqref{eq:phi0} and Young's inequality, and for $\kappa>0$
\begin{align*}
\E\bigg[\max_{1\leq m\leq M}\big|\mathscr M^1_m\big|\bigg]
&\leq\,c\,\E\bigg[\bigg(\sum_{n=1}^M\tau\zeta_R(\|\bfu_{n-1}^R\|_{W^{2,2}_x})^2\|\Phi(\bfu^R_{n-1})\|^2_{L_2(\mathfrak U,L^2_x)}\|\A\bfu^R_{n-1}\|^2_{L^2_x}\bigg)^{\frac{1}{2}}\bigg]\\ 
&\leq\,c\,\E\bigg[\bigg(\sum_{n=1}^{M}\tau\zeta_R(\|\bfu_{n-1}^R\|_{W^{2,2}_x})\big(1+\|\bfu_{n-1}^R\|_{L^2_x}^2\big)\|\A\bfu_{n-1}^R\|^2_{L^2_x}\bigg)^{\frac{1}{2}}\bigg]\\ 
&\leq\,c\,\E\bigg[\bigg(\sum_{n=1}^{M}\tau R^2\|\A\bfu_{n-1}^R\|^2_{L^2_x}\bigg)^{\frac{1}{2}}\bigg]\\ 
&\leq\,c(\delta)R^2+\,\delta\,  \E\bigg[\sum_{n=1}^{M-1}\tau\|\A\bfu_{n}^R\|_{L^2_x}^2\bigg].
\end{align*} 
Furthermore, we have
\begin{align*}
\E\bigg[\max_{1\leq m\leq M}|\mathscr M^2_{m}|\bigg]&\leq \,\delta\,\E\bigg[  \sum_{n=1}^{M}\big\|\nabla(\bfu_{n}^R-\bfu_{n-1}^R)\big\|_{L^2_x}^2\bigg]+c(\delta)\,\E\bigg[\sum_{n=1}^{M}\bigg\| \int_{t_{n-1}}^{t_n}\nabla\Phi(\bfu_{n-1}^R)\,\dd W  \bigg\|_{L^2_x}^2\bigg]\\ 
&\leq \,\delta\,\E\bigg[ \sum_{n=1}^{M}\|\nabla(\bfu_{n}^R-\bfu_{n-1}^R)\|_{L^2_x}^2 \bigg]+ c(\delta)\,\E\bigg[\tau\sum_{n=1}^M\|\Phi(\bfu_{n-1}^R)\|_{L_2(\mathfrak U;W^{1,2}_x)}^2\bigg]\\ 
&\leq \,\delta\,\E\bigg[ \sum_{n=1}^{M}\|\nabla(\bfu_{n}^R-\bfu_{n-1}^R)\|_{L^2_x}^2 \bigg]+ c(\delta)\,\E\bigg[\tau\sum_{n=1}^M\big(1+\|\bfu_{n-1}^R\|_{W^{1,2}_x}^2\big)\bigg]\\ 
&\leq \,\delta\,\E\bigg[\sum_{n=1}^{M} \|\nabla(\bfu_{n}^R-\bfu_{n-1}^R)\|_{L^2_x}^2 \bigg]+ c(\delta)
\end{align*}
due to Young's inequality, It\^{o}-isometry, \eqref{eq:phi1a} and \eqref{lem:3.1aR}. Absorbing the $\delta$-terms we conclude for $q=1$. The case $q\geq 2$ follows similarly by multiplying
with $\|\bfu_m^R\|_{W^{1,2}_x}^{2q-2}$ and iterating (see \cite[Lemma 3.1]{BCP} for details).
\end{proof}

\begin{lemma}\label{lemma:3.1B} 
Assume that $\bfu_0\in L^{2^q}(\Omega,W^{2,2}_{\Div}(\mt))$ for some $q\in\N$ and suppose that \eqref{eq:phi0}--\eqref{eq:phi2b} hold. Then the iterates $(\bfu_m^R)_{m=1}^M$ given by \eqref{tdiscrR}
satisfy the following estimates uniformly in $M$:
\begin{align}
\label{lem:3.1B}
\begin{aligned}
\E\bigg[\max_{1\leq m\leq M}\|\bfu_m^R\|^{2^q}_{W^{2,2}_x}&+\sum_{m=1}^{M}\tau\|\bfu_m^R\|_{W^{2,2}_x}^{2^q-2}\|\nabla^3\bfu_m^R\|^2_{L^2_x}\bigg]\\
&+\sum_{m=1}^{M}\|\bfu_m^R\|_{W^{2,2}_x}^{2^q-2}\|\nabla^2(\bfu_m^R-\bfu_{m-1}^R)\|^2_{L^2_x}\bigg]
\leq\,ce^{cR^{2}},
\end{aligned}
\end{align}
where $c=c(q,T,\Phi,\bfu_0)>0$ is independent of $R$.
\end{lemma}
\begin{proof}
We argue similarly to the proof of Lemma \ref{lemma:3.1} testing this time \eqref{tdiscrR}
 by $\Delta^2\bfu_m$ obtaining 
 \begin{align*}
\int_{\mt}\Delta(\bfu_{m}^R-&\bfu_{m-1}^R):\Delta\bfu_m^R\dx +\mu\tau\int_{\mt}|\Delta\nabla\bfu_m^R|^2\dx\\&=-\tau\int_{\mt}\zeta_R(\|\bfu_{m-1}^R\|_{W^{2,2}_x})(\nabla\bfu_m^R)\bfu_{m-1}^R\cdot\Delta^2\bfu^R_m\dx\\&+\zeta_R(\|\bfu_{m-1}^R\|_{W^{2,2}_x})\int_{\mt} \Delta^2\bfu_m^R\cdot\Phi(\bfu_{m-1}^R)\Delta_mW\dx.
\end{align*}
For $\delta>0$ we have by definition of $\zeta_R$ and the embeddings 
$W^{2,2}(\mt)\hookrightarrow L^\infty(\mt)$, $W^{2,2}(\mt)\hookrightarrow W^{1,4}(\mt)$
\begin{align*}
-\int_{\mt}\zeta_R&(\|\bfu_{m-1}^R\|_{W^{2,2}_x})(\nabla\bfu_m^R)\bfu_{m-1}^R\cdot\Delta^2\bfu^R_m\dx\\&=\int_{\mt}\zeta_R(\|\bfu_{m-1}^R\|_{W^{2,2}_x})(\nabla\bfu_m^R)\nabla\bfu_{m-1}^R:\Delta\nabla\bfu^R_m\dx\\
&+\int_{\mt}\zeta_R(\|\bfu_{m-1}^R\|_{W^{2,2}_x})(\nabla^2\bfu_m^R)\bfu_{m-1}^R:\Delta\nabla\bfu^R_m\dx\\
&\leq\zeta_R(\|\bfu_{m-1}^R\|_{W^{2,2}_x})\|\nabla\bfu_{m-1}^R\|_{L^4_x}\|\nabla\bfu_m^R\|_{L^4_x}\|\Delta\nabla\bfu_m^R\|_{L^2_x}\\
&+\zeta_R(\|\bfu_{m-1}^R\|_{W^{2,2}_x})\|\bfu_{m-1}^R\|_{L^\infty_x}\|\nabla^2\bfu_m^R\|_{L^2_x}\|\Delta\nabla\bfu_m^R\|_{L^2_x}
\\
&\leq\,cR\|\nabla^2\bfu_m\|_{L^2_x}\|\Delta\nabla\bfu_m\|_{L^2_x}
\\&\leq\,c(\delta)R^2\|\nabla^2\bfu_m\|_{L^2_x}^2+\delta\|\Delta\nabla\bfu_m\|_{L^2_x}^{2}.
\end{align*}
Summing up, choosing $\delta$ sufficiently small and using continuity of $\nabla^2\Delta^{-1}$ on $L^2(\mt)$ shows
 \begin{align*}
\tfrac{1}{2}\int_{\mt}|\nabla^2\bfu_{m}^R|^2\dx&+\tfrac{1}{2}\sum_{n=1}^m\int_{\mt}|\nabla^2(\bfu_{n}^R-\bfu^R_{n-1})|^2\dx +\tfrac{\mu}{2} \sum_{n=1}^m\tau\int_{\mt}|\nabla^3\bfu^R_{n}|^2\dx\\&\leq \tfrac{1}{2} \int_{\mt}|\nabla^2\bfu_{0}|^2\dx+cR^2 \sum_{n=1}^m\tau\|\nabla^2\bfu_n^R\|_{L^2_x}^2+\widetilde{\mathscr M}^1_m+\widetilde{\mathscr M}_m^2,
\end{align*}
where
\begin{align*}
\widetilde{\mathscr M}^1_m&=-\sum_{n=1}^m\zeta_R(\|\bfu_{n-1}^R\|_{W^{2,2}_x})\int_{\mt}\int_{t_{n-1}}^{t_n}\nabla\{\Phi(\bfu_{m-1}^R)\}\,\dd W\cdot \A\nabla\bfu^R_{n-1}\dx,\\
\widetilde{\mathscr M}_m^2&=\sum_{n=1}^m\zeta_R(\|\bfu_{n-1}^R\|_{W^{2,2}_x})\int_{\mt}\int_{t_{n-1}}^{t_n}\Delta\{\Phi(\bfu_{m-1}^R)\}\,\dd W\cdot \A(\bfu^R_n-\bfu^R_{n-1})\dx.
\end{align*}
By the discrete Gronwall lemma we have $\p$-a.s.
 \begin{align*}
\tfrac{1}{2}\max_{1\leq m\leq M}\int_{\mt}|\nabla^2\bfu^R_{m}|^2\dx&+\tfrac{1}{2}\sum_{n=1}^{M}\int_{\mt}|\nabla^2(\bfu^R_{n}-\bfu^R_{n-1})|^2\dx +\tfrac{\mu}{2} \sum_{n=1}^{M}\tau\int_{\mt}|\nabla^3\bfu^R_{n}|^2\dx\\&\leq ce^{cR^2}\bigg( \int_{\mt}|\nabla^2\bfu_{0}|^2\dx+\max_{1\leq m\leq M}|\widetilde{\mathscr M}^1_m|+\max_{1\leq m\leq M}|\widetilde{\mathscr M}_m^2|\bigg).
\end{align*}
We obtain further by Burkholder-Davis-Gundy inequality, \eqref{eq:phi1a} and Young's inequality, and for $\delta>0$
\begin{align*}
\E\bigg[\max_{1\leq m\leq M}\big|\widetilde{\mathscr M}^1_m\big|\bigg]
&\leq\,c\,\E\bigg[\bigg(\sum_{n=1}^M\tau\zeta_R(\|\bfu_{m-1}^R\|_{W^{2,2}_x})^2\|\Phi(\bfu^R_{m-1})\|^2_{L_2(\mathfrak U,W^{1,2}_x)}\|\A\nabla\bfu^R_{n-1}\|^2_{L^2_x}\ds\bigg)^{\frac{1}{2}}\bigg]\\ 
&\leq\,c\,\E\bigg[\bigg(\sum_{n=1}^{M}\tau\zeta_R(\|\bfu_{m-1}^R\|_{W^{2,2}_x})\big(1+\|\bfu_{m-1}^R\|_{W^{1,2}_x}^2\big)\|\A\nabla\bfu_{n-1}^R\|^2_{L^2_x}\bigg)^{\frac{1}{2}}\bigg]\\ 
&\leq\,c\,\E\bigg[\bigg(\sum_{n=1}^{M}\tau R^2\|\A\nabla\bfu_{n-1}^R\|^2_{L^2_x}\bigg)^{\frac{1}{2}}\bigg]\\ 
&\leq\,c(\delta)R^2+\,\delta \E\bigg[\sum_{n=1}^{M-1}\tau\|\A\nabla\bfu_{n}^R\|_{L^2_x}^2\bigg].
\end{align*} 
Furthermore, we have
\begin{align*}
\E\bigg[\max_{1\leq m\leq M}|\widetilde{\mathscr M}^2_{m}|\bigg]&\leq \,\delta\,\E\bigg[  \sum_{n=1}^{M}\big\|\nabla^2(\bfu_{n}^R-\bfu_{n-1}^R)\big\|_{L^2_x}^2\bigg]\\&+c(\delta)\,\E\bigg[\sum_{n=1}^{M}\zeta_R(\|\bfu_{m-1}^R\|_{W^{2,2}_x})^2\bigg\| \int_{t_{n-1}}^{t_n}\Delta\Phi(\bfu_{n-1}^R)\,\dd W  \bigg\|_{L^2_x}^2\bigg]\\ \nonumber
&\leq \,\delta\,\E\bigg[ \sum_{n=1}^{M}\|\nabla^2(\bfu_{n}^R-\bfu_{n-1}^R)\|_{L^2_x}^2 \bigg]\\&+ c(\delta)\,\E\bigg[\tau\sum_{n=1}^M\zeta_R(\|\bfu_{m-1}^R\|_{W^{2,2}_x})\|\Phi(\bfu_{n-1}^R)\|_{L_2(\mathfrak U;W^{2,2}_x)}^2\bigg]\\ 
&\leq \,\delta\,\E\bigg[ \sum_{n=1}^{M}\|\nabla^2(\bfu_{n}^R-\bfu_{n-1}^R)\|_{L^2_x}^2 \bigg]\\&+ c(\delta)\,\E\bigg[\tau\sum_{n=1}^M\zeta_R(\|\bfu_{m-1}^R\|_{W^{2,2}_x})\big(1+\|\bfu_{n-1}^R\|_{W^{1,4}_x}^4+\|\bfu_{n-1}^R\|_{W^{2,2}_x}^2\big)\bigg]\\ 
&\leq \,\delta\,\E\bigg[\sum_{n=1}^{M} \|\nabla^2(\bfu_{n}^R-\bfu_{n-1}^R)\|_{L^2_x}^2 \bigg]+ c(\delta) R^2\E\bigg[\tau\sum_{n=1}^{M} \|\nabla^2\bfu_{n-1}^R\|_{L^2_x}^2 \bigg]
\end{align*}
due to Young's inequality, It\^{o}-isometry and \eqref{eq:phi2a} (together with the embedding $W^{2,2}(\mt)\hookrightarrow W^{1,4}(\mt)$). Absorbing the $\delta$-terms and applying Lemma \ref{lemma:3.1} yields the claim for $q=1$, whereas the general case follows again by iteration.
\end{proof}
 
 For $R>0$, we define the (discrete) $(\mathfrak{F}_{t_{m}})$-stopping time
\begin{align} \label{eq:jR} 
{\mathfrak s}_{R}^{\tt d} &:= \min_{0 \leq m \leq M} \biggl\{ t_m:\ \max_{0\leq n\leq m} \Vert  {\bf u}_{n}\Vert_{W^{2,2}_x} \geq R\biggr\},
\end{align}
where we set ${\mathfrak s}_{R}^{\tt d}=t_M$ if the set above is empty.
 Note that ${\mathfrak s}_{R}^{\tt d} \in \{t_m\}_{m=0}^M$, 
 with random index 
 ${\mathfrak j}_{R} \in {\mathbb N}_0 \cap [0,M]$, such that ${\mathfrak s}_{R}^{\tt d} = t_{{\mathfrak j}_{R}}$. The crucial point is now to show a counterpart of the strict positivity
 of the stopping time from the continuous solution, cf. Definition
 \ref{def:strsol}. This is the content of the next lemma, which states that ${\mathfrak s}_{R}^{\tt d}\geq \tau$ with high probability.
 \begin{lemma}\label{lem:sgeqtau}
 Suppose that the assumptions from Lemma \ref{lemma:3.1B}
 (with $q=1$) hold and let $R=R(\tau)$ be chosen such that
 $\tau e^{cR^2}\rightarrow0$ as $\tau\rightarrow0$. Then we have for any $\ell\in\N$
 \begin{align}\label{eq:sgeqtau}
 \lim_{\tau\rightarrow0}\p\big([{\mathfrak s}_{R}^{\tt d}\leq\ell\tau]\big)=0.
 \end{align}
 \end{lemma}
 \begin{proof}
 Arguing as in the proof of Lemma \ref{lemma:3.1B}
 we have
  \begin{align*}
\tfrac{1}{2}\max_{1\leq m\leq \ell}\int_{\mt}|\nabla^2\bfu^R_{m}|^2\dx&+\tfrac{1}{2}\sum_{n=1}^{\ell}\int_{\mt}|\nabla^2(\bfu^R_{n}-\bfu^R_{n-1})|^2\dx +\tfrac{\mu}{2} \sum_{n=1}^{\ell}\tau\int_{\mt}|\nabla^3\bfu^R_{n}|^2\dx\\&\leq \tfrac{1}{2} \int_{\mt}|\nabla^2\bfu_{0}|^2\dx+\max_{1\leq m\leq \ell}|\widetilde{\mathscr M}^1_m|+\max_{1\leq m\leq \ell}|\widetilde{\mathscr M}_m^2|\\
&+cR^2 \sum_{m=1}^\ell\tau\int_{\mt}|\nabla^2\bfu_m^R|^2\dx.
\end{align*}
where the expectation of the last term can be estimated by
$c\ell\tau e^{cR^2}$
 as a consequence of Lemma \ref{lemma:3.1B}.
For the stochastic terms we have again
\begin{align*}
\E\bigg[\max_{1\leq m\leq \ell}\big|\widetilde{\mathscr M}^1_m\big|\bigg]
&\leq\,c\,\E\bigg[\bigg(\sum_{n=1}^{\ell}\tau R^2\|\A\nabla\bfu_{n-1}^R\|^2_{L^2_x}\bigg)^{\frac{1}{2}}\bigg]\\ 
&\leq\,c(\delta)\tau \ell R^2+\,\delta  \E\bigg[\sum_{n=1}^{\ell-1}\tau\|\A\nabla\bfu_{n}^R\|_{L^2_x}^2\bigg],
\end{align*} 
\begin{align*}
\E\bigg[\max_{1\leq m\leq \ell}|\widetilde{\mathscr M}^2_{m}|\bigg]
&\leq \,\delta\,\E\bigg[ \sum_{m=1}^{\ell}\|\nabla^2(\bfu_{m}^R-\bfu_{m-1}^R)\|_{L^2_x}^2 \bigg]\\&+ c(\delta)\,\E\bigg[\tau\sum_{n=1}^\ell\zeta_R(\|\bfu_{m-1}^R\|_{W^{2,2}_x})\big(1+\|\bfu_{n-1}^R\|_{W^{1,4}_x}^4+\|\bfu_{n-1}^R\|_{W^{2,2}_x}^2\big)\bigg]\\ 
&\leq \,\delta\,\E\bigg[\sum_{n=1}^{\ell} \|\nabla^2(\bfu_{n}^R-\bfu_{n-1}^R)\|_{L^2_x}^2 \bigg]+ c(\delta)\tau\ell R^2.
\end{align*}
Absorbing the $\delta$-terms we conclude 
\begin{align*}
\E\bigg[\max_{1\leq m\leq\ell}\|\bfu_m^R\|_{W^{2,2}_x}^2\bigg]\leq\, \E\bigg[\int_{\mt}|\nabla^2\bfu_{0}|^2\dx\bigg]
+c\tau\ell e^{ cR^2}
\end{align*}
such that
 \begin{align*}
\p\big([{\mathfrak s}_{R}^{\tt d}\leq \ell\tau]\big)&=\p\bigg[\max_{1\leq m\leq\ell}\|\bfu_m\|_{W^{2,2}_x}^2\geq R^2\bigg]\\
&\leq\p\bigg[\max_{1\leq m\leq\ell}\|\bfu_m^R\|_{W^{2,2}_x}^2\geq R^2\bigg]\\
&\leq\, \frac{1}{R^2}\E\bigg[\int_{\mt}|\nabla^2\bfu_{0}|^2\dx\bigg]
+c\tau\ell e^{ cR^2}
 \end{align*}
 by Markov's inequality. By assumption the right-hand sides vanishes as $\tau\rightarrow0$.
 \end{proof}

With relation \eqref{eq:sgeqtau} at hand it is now meaningful to transfer the estimates for $(\bfu_m^R)_{m=1}^M$ from Lemma  \ref{lemma:3.1a} and  \ref{lemma:3.1B} to 
$(\bfu_m^R)_{m=1}^M$. Noticing that $\bfu_m=\bfu_m^R$
in $[{\mathfrak s}_{R}^{\tt d}\geq t_m]$ we obtain the following corollary.
\begin{corollary}\label{cor:3.1} 
Assume that $q\in\N$ and that \eqref{eq:phi0} holds. Then the iterates $(\bfu_m)_{m=1}^M$ given by \eqref{tdiscr}
satisfy the following estimates uniformly in $M$:
\begin{enumerate}
\item Suppose that $\bfu_0\in L^{2^q}(\Omega,W^{1,2}_{\Div}(\mt))$ and that additionally \eqref{eq:phi1a} and \eqref{eq:phi1b} hold. Then we have
\begin{align*}
\E\bigg[\max_{1\leq m\leq {\mathfrak j}_{R} }\|\bfu_m\|^{2^q}_{W^{1,2}_x}&+\sum_{m=1}^{{\mathfrak j}_{R} }\tau\|\bfu_m\|_{W^{1,2}_x}^{2^q-2}\|\nabla^2\bfu_m\|^2_{L^2_x}\bigg]\\
&+\sum_{m=1}^{\mathfrak j_R}\|\bfu_m^R\|_{W^{2,2}_x}^{2^q-2}\|\nabla(\bfu_m-\bfu_{m-1})\|^2_{L^2_x}\bigg]\leq\,ce^{cR^{2}}.
\end{align*}
\item Suppose that $\bfu_0\in L^{2^q}(\Omega,W^{2,2}_{\Div}(\mt))$  and that additionally \eqref{eq:phi1a}--\eqref{eq:phi2b} hold. Then we have
\begin{align*}
\E\bigg[\max_{1\leq m\leq {\mathfrak j}_{R} }\|\bfu_m\|^{2^q}_{W^{2,2}_x}&+\sum_{m=1}^{{\mathfrak j}_{R} }\tau\|\bfu_m\|_{W^{2,2}_x}^{2^q-2}\|\nabla^3\bfu_m\|^2_{L^2_x}\bigg]\\
&+\sum_{m=1}^{\mathfrak j_R}\|\bfu_m^R\|_{W^{2,2}_x}^{2^q-2}\|\nabla^2(\bfu_m-\bfu_{m-1})\|^2_{L^2_x}\bigg]\leq\,ce^{cR^{2}}.
\end{align*}
\end{enumerate}
Here $c=c(q,T,\Phi,\bfu_0)>0$ is independent of $R$.
\end{corollary}
\begin{remark}
In the estimates from Lemmas \ref{lemma:3.1} and \ref{lemma:3.1B} it is also possible to control higher moments provided the corresponding moments are bounded for the initial datum. This transfers also to Corollary \ref{cor:3.1} and, in particular, implies that
\begin{align}\label{eq:1007}
\begin{aligned}
\E\bigg[\sum_{m=1}^{\mathfrak j_R}\|\nabla(\bfu_m-\bfu_{m-1})\|^2_{L^2_x}\bigg]^q\leq\,ce^{cR^{2}},
\end{aligned}
\end{align}
provided we have $\bfu_0\in L^{2^q}(\Omega,W^{1,2}_{\Div}(\mt))$.
\end{remark}

Now we are going to introduce the pressure function for \eqref{tdiscr}.
 For $\bfphi\in W^{1,2}(\mt)$ we can insert $\bfphi-\nabla\Delta^{-1}\Div\bfphi\in W^{1,2}_{\Div}(\mt)$ in \eqref{tdiscr}
and obtain
\begin{align}
\nonumber
\int_{\mt}\bfu_m\cdot\bfvarphi\dx &-\tau\bigg(\int_{\mt}\bfu_m\otimes\bfu_{m-1}:\nabla\bfphi\dx+\mu\int_{\mt}\nabla\bfu_{m}:\nabla\bfphi\dx\bigg)\\
\label{eq:pressuret}&=\int_{\mt}\bfu_{m-1}\cdot\bfvarphi\dx
+\Delta t\int_{\mt}\pi_m^{\mathrm{det}}\,\Div\bfphi\dx
\\
\nonumber&+\int_{\mt}\Phi(\bfu_{m-1})\,\Delta_m W\cdot \bfvarphi\dx+\int_{\mt}\int_0^t\Phi^\pi_{m-1}\,\Delta_m W\cdot \bfvarphi\dx,
\end{align}
where 
\begin{align*}
\pi_m^{\mathrm{det}}&=-\Delta^{-1}\Div\Div\big(\bfu_m\otimes\bfu_{m-1}\big),\\
\Phi_{m-1}^\pi&=-\nabla\Delta^{-1}\Div\Phi(\bfu_{m-1}).
\end{align*}
Similar to \cite[Lemma 3 \& 4]{BrDo} we give some estimates for $\pi_m^{\mathrm{det}}$ and $\Phi_{m-1}^\pi$.
\begin{lemma}\label{lemma:3.2}
Assume that $\bfu_0\in L^{4}(\Omega,W^{1,2}_{div}(\mt))$ and that $\Phi$ satisfies \eqref{eq:phi0}--\eqref{eq:phi1b}.
For all $m\in\{1,...,\mathfrak j_R\}$ the random variable $\pi_m^{\mathrm{det}}$ is $\mathfrak F_{t_m}$-measurable, has values in $W^{1,2}(\mt)$ and we have uniformly in $\tau$
\begin{align*}
\E\bigg[\tau \sum_{m=1}^{\mathfrak j_R}\big\|\nabla \pi_m^{\mathrm{det}}\big\|_{L^2_x}^2\bigg]\leq \,ce^{cR^2},
\end{align*}
where $c=c(T,\Phi,\bfu_0)$ is independent of $R$. 
\end{lemma}
\begin{proof}
The $\mathfrak F_{t_m}$-measurability of $\pi_m^{\mathrm{det}}$ follows directly from the measurability of $\bfu_m$.
By continuity of the operator $\nabla\Delta^{-1}\Div$ on $L^2(\mt)$, we have
\begin{align*}
\big\|\nabla \pi_m^{\mathrm{det}}\big\|_{L^2_x}^2&\leq\,c\,\big\|\Div(\bfu_m\otimes\bfu_m)\big\|^2_{L^2_x}\leq\,c\,\|\bfu_m\|_{{L^4_x}}^2\|\nabla\bfu_m\|_{{L^4_x}}^2\leq\,c\,\|\bfu_m\|_{{W^{1,2}_x}}^2\|\nabla^2\bfu_m\|^2_{{W^{2,2}_x}}
\end{align*}
$\p$-a.s., also making use of Sobolev's embedding in three dimensions. Now, summing with respect to $m$, applying expectations and using Corollary \ref{cor:3.1} (a) with $q=2$ yields the claim.
\end{proof}
\begin{lemma}\label{lemma:3.3}
Assume that $\bfu_0\in L^{2}(\Omega,L^{2}_{div}(\mt))$ and that $\Phi$ satisfies \eqref{eq:phi0}--\eqref{eq:phi1a}.
For all $m\in\{1,...,M\}$ the random variable $\Phi^\pi_{m}$ is $\mathfrak F_{t_m}$-measureable, has values in $L_2(\mathfrak U;W^{1,2}(\mt))$ and we have
uniformly in $\tau$
\begin{align*}
\E\bigg[\tau \sum_{m=1}^{M}\big\|\Phi^\pi_m\big\|_{L_2(\mathfrak U;W^{1,2}_x)}^2\bigg]\leq \,c
\end{align*}
where $c=c(T,\Phi,\bfu_0)$. 
\end{lemma}
\begin{proof}
As with Lemma \ref{lemma:3.2}, the proof mainly relies on the continuity of $\nabla\Delta^{-1}\Div$ on $L^2(\mt)$. Here, we have by \eqref{eq:phi1a}
\begin{align*}
\|\Phi^\pi_m\|^2_{L_2(\mathfrak U;W^{1,2}_x)}&=\sum_{k\geq1}\|\nabla\Delta^{-1}\Div\big(\Phi(\bfu_m)e_k\big)\|^2_{W^{1,2}_x}\\
&\leq\,c\,\sum_{k\geq1}\|\Phi(\bfu_m)e_k\|^2_{W^{1,2}_x}=c\,\|\Phi(\bfu_m)\|^2_{L_2(\mathfrak U;W^{1,2}_x)}\leq\,c\,\big(1+\|\nabla\bfu_m\|^2_{W^{1,2}_x}\big).
\end{align*}
Summing over $m$, applying expectations and using Lemma \ref{lemma:3.1a}
finishes the proof.
\end{proof}

\subsection{Temporal error analysis}
\label{sec:terror}
For every $m\geq 1$ introduce
 the discrete stopping time
\begin{align}\label{eq:mR}\mathfrak t_m^R:=\max_{1\leq n\leq m}\big\{t_n:t_n\leq \mathfrak t_R\big\},
\end{align}
which is obviously $\mathfrak F_{t_m}$-measurable. 
Furthermore, we define $\mathfrak m_R$ as the unique index in $\{1,2,\dots,M\}$ such that $\mathfrak t_M^R=t_{\mathfrak m_R}$.
Our main effort in this section is devoted to the proof of the following theorem.

\begin{theorem}\label{thm:4t}
Let $\bfu_0\in L^8(\Omega,W^{2,2}_{\Div}(\mt))$ be $\F_0$-measurable and assume that $\Phi$ satisfies \eqref{eq:phi0}--\eqref{eq:phi2b}. Let $$(\bfu,(\mathfrak{t}_R)_{R\in\N},\mathfrak{t})$$ be the unique maximal global strong solution to \eqref{eq:SNS} in the sense of Definition \ref{def:maxsol}.
Then we have for all $R\in\N$ and all $\alpha<\frac{1}{2}$
\begin{align}\label{eq:thm:4t}
\begin{aligned}
\E\bigg[\max_{1\leq m\leq \mathfrak m_R}\|\bfu(t_m)-\bfu_{m}\|_{L^2_x}^2&+\sum_{m=1}^{\mathfrak m_R} \tau\|\nabla\bfu(t_{m})-\nabla\bfu_{m}\|_{L^2_x}^2\bigg)\bigg]\leq \,ce^{cR^2}\,\tau^{2\alpha},
\end{aligned}
\end{align}
where $(\bfu_{m})_{m=1}^M$ is the solution to \eqref{tdiscr}.
The constant $c$ in \eqref{eq:thm:4t} is independent of $\tau$ and $R$.
\end{theorem}

Our main result on the temporal error is now a direct consequence of Theorem \ref{eq:thm:4t}: Supposing that $R= R(\tau)\leq  c^{-1/2}\sqrt{-2\varepsilon \log \tau}$ as $\tau\rightarrow0$ (note that this includes, in particular, any choice of fixed $R\in\N$), where $\varepsilon>0$ is arbitrary, 
and relabelling $\alpha$ we have proved the following result.

\begin{theorem}\label{thm:maint}
Let $\bfu_0\in L^8(\Omega,W^{2,2}_{\Div}(\mt))$ be $\F_0$-measurable and assume that $\Phi$ satisfies \eqref{eq:phi0}--\eqref{eq:phi2b}. Let $$(\bfu,(\mathfrak{t}_R)_{R\in\N},\mathfrak{t})$$ be the unique maximal global strong solution to \eqref{eq:SNS} from Theorem \ref{thm:inc2d}.
Then we have for any $\xi>0$, $\alpha<\frac{1}{2}$, 
\begin{align*}
&\mathbb P\bigg[\max_{1\leq m\leq \mathfrak m_{R}}\|\bfu(t_m)-\bfu_{m}\|_{L^2_x}^2+\sum_{m=1}^{\mathfrak m_{R}} \tau\|\nabla\bfu(t_m)-\nabla\bfu_{m}\|_{L^2_x}^2>\xi\,\tau^{2\alpha}\bigg]\rightarrow 0
\end{align*}
as $\tau\rightarrow0$,
where $(\bfu_{m})_{m=1}^M$ is the solution to \eqref{tdiscr}.
\end{theorem}

\begin{proof}[Proof of Theorem \ref{thm:4t}]
Define the error $\bfe_{m}=\bfu(t_m)-\bfu_{m}$ for $m\in\{0,1,\dots,\mathfrak m_R\}$. Subtracting \eqref{eq:mom} and \eqref{tdiscr} and using that $t_m\leq \mathfrak t_R$ for $m\leq\mathfrak m_R$ we obtain
\begin{align*}
\int_{\mt}&\bfe_{m}\cdot\bfvarphi \dx +\mu \int_{ t_{m-1}}^{ t_m}\int_{\mt}\nabla\bfu(\sigma):\nabla\bfphi\dx\ds-\mu\tau\int_{\mt}\nabla\bfu_{m}:\nabla\bfphi\dx\\&=\int_{\mt}\bfe_{m-1}\cdot\bfvarphi \dx+\tau\int_{\mt}(\nabla\bfu_{m})\bfu_{m-1}\cdot\bfphi\dx-\int_{ t_{m-1}}^{t_m}\int_{\mt}(\nabla\bfu(\sigma))\bfu(\sigma)\,\cdot\bfphi\dxs\\
&+\int_{\mt}\int_{t_{m-1}}^{ t_m}\Phi(\bfu(\sigma))\,\dd W\cdot \bfvarphi\dx-\int_{\mt}\int_{t_{m-1}}^{t_m}\Phi(\bfu_{m-1})\,\dd W\cdot \bfvarphi\dx
\end{align*}
for every $\bfphi\in W^{1,2}_{\Div}(\mt)$.
Setting $\bfphi=\bfe_{h,m}$ and applying the identity $\bfa\cdot(\bfa-\bfb)=\frac{1}{2}\big(|\bfa|^2-|\bfb|^2+|\bfa-\bfb|^2\big)$ (which holds for any $\bfa,\bfb\in\mathbb R^3$) we gain
\begin{align*}
\int_{\mt}&\frac{1}{2}\big(|\bfe_{m}|^2-|\bfe_{m-1}|^2+|\bfe_{m}-\bfe_{m-1}|^2\big) \dx+\mu \tau\int_{\mt}|\nabla\bfe_{m}|^2\dx\\
&=\mu\int_{t_{m-1}}^{t_{m}}\int_{\mt}\big(\nabla\bfu( t_{m})-\nabla\bfu(\sigma)\big):\nabla\bfe_{m}\dx\ds
\\
&+\int_{t_{m-1}}^{t_{m}}\int_{\mt}\Big((\nabla\bfu(t_{m}))\bfu(t_{m-1})-(\nabla\bfu(\sigma))\bfu(\sigma)\Big)\cdot\bfe_{m}\dxs\\
&-\tau\int_{\mt}\Big((\nabla\bfu(t_{m}))\bfu(t_{m-1})-(\nabla\bfu_{m})\bfu_{m-1}\Big)\cdot\bfe_{m}\dx\\
&+\int_{\mt}\int_{t_{m-1}}^{t_m}\big(\Phi(\bfu(\sigma))-\Phi(\bfu_{m-1})\big)\,\dd W\cdot \bfe_{m}\dx\\
&=:I_1(m)+\dots +I_5(m).
\end{align*}
Eventually, we will take the maximum with respect to $m\in\{1,\dots,\mathfrak m_R\}$ and apply expectations. Let us explain how to deal  with $\E\big[\max_m I_1(m)],\dots ,\E[\max_mI_5(m)]$ independently.

We have 
\begin{align*}
I_1(m)&\leq \,\kappa\tau\int_{\mt}|\nabla\bfe_{m}|^2\dx+c(\kappa)\int_{t_{m-1}}^{t_{m}}\int_{\mt}|\nabla(\bfu(t_{m})-\bfu(\sigma)|^2\dx\ds\\
&\leq \,\kappa\tau\int_{\mt}|\nabla\bfe_{m}|^2\dx+c(\kappa)\tau^{1+2\alpha}\|\nabla\bfu\|_{C^\alpha([t_{m-1},t_{m}];L^2_x)}^2,
\end{align*}
where the expectation of the last term can be controlled for $m\leq\mathfrak m_R$ by $\tau^{2\alpha+1}R^{12}$ using Corollary \ref{cor:uholder} and $\mathfrak t_{m}^R\leq\mathfrak t_R$.
We proceed by
\begin{align*}
I_2(m)&=-\int_{ t_{m-1}}^{t_{m}}\int_{\mt}\Big(\bfu(t_{m})\otimes\bfu( t_{m-1})-\bfu(\sigma)\otimes\bfu(\sigma)\Big):\nabla\bfe_{m}\dxs\\
&\leq \,\kappa\tau\int_{\mt}|\nabla\bfe_{m}|^2\dx+c(\kappa)\int_{ t_{m-1}}^{t_{m}}\int_{\mt}|\bfu(t_{m})\otimes\bfu(t_{m-1})-\bfu(\sigma)\otimes\bfu(\sigma)|^2\dx\ds\\
&\leq \,\delta\tau\int_{\mt}|\nabla\bfe_{m}|^2\dx+c(\delta)\tau^{1+2\alpha}\|\bfu\|^2_{L^\infty(( t_{m-1}, t_{m})\times\mt)}\|\bfu\|_{C^\alpha([t_{m-1},t_m];L^2_x)}^2\\
&\leq \,\delta\tau_{m}^R\int_{\mt}|\nabla\bfe_{m}|^2\dx+c(\delta)\tau^{1+2\alpha}R^2\|\bfu\|_{C^\alpha([t_{m-1},t_m];L^2_x)}^2
\end{align*}
for $m\leq \mathfrak m_R$ using the embedding $W^{2,2}(\mt)\hookrightarrow L^\infty(\mt)$ and $\mathfrak t_{m}^R\leq\mathfrak t_R$. 
We rewrite $I_3(m)$ as
\begin{align*}
I_3(m)&=-\tau\int_{\mt}(\nabla\bfe_{m})\bfe_{m-1}\cdot\bfu( t_{m})\dx
\end{align*}
and obtain for any $\delta>0$ and $m\leq\mathfrak m_R$
\begin{align*}
I_3(m)&\leq \tau\|\nabla\bfe_{m}\|_{L^2_x}\|\bfe_{m-1}\|_{L^2_x}\|\bfu(t_{m})\|_{L^\infty_x}\\
&\leq\,\delta\tau\|\nabla\bfe_{m}\|^2_{L^2_x}+c(\delta)\, R^2\|\bfe_{m-1}\|^2_{L^2_x}.
\end{align*}
The last term will be dealt with by Gronwall's lemma leading to a constant of the form $c e^{cR^2}$. 

In order to estimate the stochastic term $I_5$ we write
\begin{align*}
\mathscr M_{m,1}&=\sum_{n=1}^mI_5(n)=
\sum_{n=1}^m\int_{\mt}\int_{t_{n-1}}^{t_{n}}\big(\Phi(\bfu)-\Phi(\bfu_{n-1})\big)\,\dd W\cdot \bfe_{n}\dx\\
&= \sum_{n=1}^m\int_{\mt}\int_{t_{n-1}}^{ t_{n}}\big(\Phi(\bfu)-\Phi(\bfu_{n-1})\big)\,\dd W\cdot \bfe_{n-1}\dx\\
&+ \sum_{n=1}^m\int_{\mt}\int_{t_{n-1}}^{ t_{n}}\big(\Phi(\bfu)-\Phi(\bfu_{n-1})\big)\,\dd W\cdot \Pi_h(\bfe_{n}-\bfe_{n-1})\dx\\
&= \int_{0}^{ t_{m}}\sum_{n=1}^M\mathbf1_{[t_{n-1},t_n)}\int_{\mt}\big(\Phi(\bfu)-\Phi(\bfu_{n-1})\big)\,\dd W\cdot \bfe_{n-1}\dx\\
&+ \sum_{n=1}^m\int_{\mt}\int_{t_{n-1}}^{ t_{n}}\big(\Phi(\bfu)-\Phi(\bfu_{n-1})\big)\,\dd W\cdot (\bfe_{n}-\bfe_{n-1})\dx\\
&=:\mathscr M_{1}^1(t_{m})+\mathscr M_{m,1}^2.
\end{align*}
Since the process $(\mathscr M_{1}^1(t\wedge\mathfrak t_R))_{t\geq0}$ is an $(\mathfrak F_t)$-martingale, through the use of the Burkholder-Davis-Gundy inequality (using that $\mathfrak t_M^R\leq \mathfrak t_R$ by definition) we see that
\begin{align*}
&\E\bigg[\max_{1\leq m\leq \mathfrak m_R}\big|\mathscr M_{1}^1(t_m)\big|\bigg]\leq \E\bigg[\sup_{s\in[0,\mathfrak t_M^R]}\big|\mathscr M_{1}^1(s)\big|\bigg]\leq \E\bigg[\sup_{s\in[0,T]}\big|\mathscr M_{1}^1(s\wedge \mathfrak t_R)\big|\bigg]\\
&\leq\,c\,\E\bigg[\int_{0}^{T \wedge\mathfrak t_{R}}\sum_{n=1}^M\mathbf1_{[t_{n-1},t_n)}\|\Phi(\bfu)-\Phi(\bfu_{n-1})\|^2_{L_2(\mathfrak U,L^2_x)}\|\bfe_{n-1}\|^2_{L^2_x}\dt\bigg]^{\frac{1}{2}}\\
&\leq\,c\,\E\bigg[\max_{1\leq n\leq \mathfrak m_R}\|\bfe_{n}\|_{L^2_x}\bigg(\int_{0}^{T\wedge\mathfrak t_R}\sum_{n=1}^M\mathbf1_{[t_{n-1},t_n)}\|\Phi(\bfu)-\Phi(\bfu_{n-1})\|^2_{L_2(\mathfrak U,L^2_x)}\dt\bigg)^{\frac{1}{2}}\bigg]\\
&\leq\,\delta\,\E\bigg[\max_{1\leq n\leq \mathfrak m_R}\|\bfe_{n}\|^2_{L^2_x}\bigg]+\,c(\delta)\,\E\bigg[\int_{0}^{T\wedge\mathfrak t_R}\sum_{n=1}^M\mathbf1_{[t_{n-1},t_n)}\|\bfu-\bfu_{n-1}\|_{L^2_x}^2\dt\bigg]\\
&\leq\,\delta\,\E\bigg[\max_{1\leq n\leq \mathfrak m_R}\|\bfe_{n}\|^2_{L^2_x}\bigg]+\,c(\delta)\,\E\bigg[\int_{0}^{T\wedge\mathfrak t_R}\|\bfu-\bfu( t_{n-1})\|_{L^2_x}^2\dt\bigg]\\
&+\,c(\delta)\,\E\bigg[\int_{0}^{T\wedge\mathfrak t_R}\sum_{n=1}^M\mathbf1_{[t_{n-1},t_n)}\|\bfe_{n-1}\|_{L^2_x}^2\dt\bigg].
\end{align*}
Here, we also used \eqref{eq:phi0} as well as Young's inequality for arbitrary $\delta>0$. Finally, we can control the last term by
\begin{align*}
\E\bigg[\int_{0}^{T\wedge\mathfrak t_R}\sum_{n=1}^M\mathbf1_{[t_{n-1},t_n)}\|\bfe_{n-1}\|_{L^2_x}^2\dt\bigg]&\leq \E\bigg[\sum_{n=1}^{\mathfrak m_R+1}\tau\|\bfe_{n-1}\|_{L^2_x}^2\dt\bigg]\\
&\leq \E\bigg[\sum_{n=0}^{\mathfrak m_R}\tau\|\bfe_{n}\|_{L^2_x}^2\dt\bigg]
\end{align*}
since $\mathfrak t_R\wedge t_M\leq \mathfrak t^R_{M+1}$.
Applying \eqref{eq:stab'} as well as Lemma \ref{lem:reg} (b) and Corollary \ref{cor:uholder} (b), we obtain
\begin{align*}
\E\bigg[\max_{1\leq m\leq \mathfrak m_R}\big|\mathscr M_{1}^1(t_m)\big|\bigg]&\leq\,\delta\,\E\bigg[\max_{1\leq n\leq \mathfrak m_R}\|\bfe_{n}\|^2_{L^2_x}\bigg]+\,c(\delta)\,\E\bigg[\sum_{n=0}^{\mathfrak m_R} \tau\|\bfe_{n}\|_{L^2_x}^2\bigg]\\
&+c(\delta)\tau^{2\alpha}\E\big[\|\bfu\|_{C^\alpha([0,T\wedge\mathfrak t_R],L^2_x)}^2\big]\\
&\leq\,\delta\,\E\bigg[\max_{1\leq n\leq \mathfrak m_R}\|\bfe_{n}\|^2_{L^2_x}\bigg]+\,c(\delta)\,\E\bigg[\sum_{n=0}^{\mathfrak m_R} \tau\|\bfe_{n}\|_{L^2_x}^2\bigg]+c(\delta)\tau^{2\alpha}R^4.
\end{align*}
As far as $\mathscr M_{m,1}^2$ is concerned we argue similarly. Using the Cauchy-Schwarz inequality, Young's inequality, It\^{o}-isometry and \eqref{eq:phi0} we have
\begin{align*}
&\E\bigg[\max_{1\leq m\leq \mathfrak m_R}|\mathscr M_{m,1}^2|\bigg]\\&\leq \E\bigg[ \sum_{n=1}^{\mathfrak m_R}\bigg( \delta \|\bfe_{n}-\bfe_{n-1}\|_{L^2_x}^2 +c(\delta) \left\| \int_{ t_{n-1}}^{ t_{n}}\big(\Phi(\bfu)-\Phi(\bfu_{n-1})\big)\,\dd W  \right\|_{L^2_x}^2\bigg) \bigg]\\
&\leq \delta\E\bigg[ \sum_{n=1}^{\mathfrak m_R} \|\bfe_{n}-\bfe_{n-1}\|_{L^2_x}^2 \bigg] + c(\delta)\,\E\bigg[\sum_{n=1}^{\mathfrak m_R}\int_{t_{n-1}}^{ t_{n}}\|\bfu-\bfu_{n-1}\|_{L^2_x}^2\dt\bigg]\\
&\leq \delta\E\bigg[ \sum_{n=1}^{\mathfrak m_R} \|\bfe_{n}-\bfe_{n-1})\|_{L^2_x}^2 \bigg] +c(\delta) \,\E\bigg[\sum_{n=1}^{\mathfrak m_R} \int_{ t_{n-1}}^{ t_{n}}\|\bfu-\bfu(t_{n-1})\|_{L^2_x}^2\dt\bigg]\\&+\,c(\delta)\,\E\bigg[\sum_{n=1}^{\mathfrak m_R} \tau\|\bfe_{n-1}\|_{L^2_x}^2\bigg]\\
&\leq \delta\E\bigg[ \sum_{n=1}^{\mathfrak m_R} \|\bfe_{n}-\bfe_{n-1}\|_{L^2_x}^2 \bigg] +\,c(\delta)\,\E\bigg[\sum_{n=1}^{\mathfrak m_R} \tau\|\bfe_{n-1}\|_{L^2_x}^2\bigg]+c(\delta)\tau^{2\alpha}R^{12}
\end{align*}
as a consequence of Lemma \ref{lem:reg} (b) (using also \eqref{eq:stab}) of Corollary \ref{cor:uholder} (b).
Collecting all estimates, choosing $\delta$ small enough and applying Gronwall's lemma yields the claim.
\end{proof}

\section{Space-time discretisation}
\label{sec:txerror}
Now we consider a fully practical scheme combining the implicit Euler scheme in time (as in the last section) with a finite element approximation in space. 
For a given $h>0$, let $\bfu_{h,0}$ be an $\mathfrak F_0$-mesurable random variable with values in $V^h_{\Div}(\mt)$ (for instance $\Pi_h\bfu_0$). We aim at constructing iteratively a sequence of random variables $\bfu_{h,m}$ with values in $V^h_{\Div}(\mt)$ such that
for every $\bfphi\in V^h_{\Div}(\mt)$ it holds true $\p$-a.s.
\begin{align}\label{txdiscr}
\begin{aligned}
\int_{\mt}&\bfu_{h,m}\cdot\bfvarphi \dx +\Delta t\int_{\mt}\big((\nabla\bfu_{h,m})\bfu_{h,m-1}+(\Div\bfu_{h,m-1})\bfu_{h,m}\big)\cdot\bfphi\dx\\
&+\mu\,\Delta t\int_{\mt}\nabla\bfu_{m}:\nabla\bfphi\dx=\int_{\mt}\bfu_{h,m-1}\cdot\bfvarphi \dx+\int_{\mt}\Phi(\bfu_{h,m-1})\,\Delta_mW\cdot \bfvarphi\dx,
\end{aligned}
\end{align}
where $\Delta_m W=W(t_m)-W(t_{m-1})$. The existence of iterates $(\bfu_{h,m})_{m=1}^M$ given by \eqref{txdiscr}
which are $\mathfrak F_{t_m}$-measurable is shown in \cite[Lemma 3.1]{BCP}. Furthermore, it holds for $q\in\N$
\begin{align}
\label{lem:4.1}\E\bigg[\max_{1\leq m\leq M}\|\bfu_{h,m}\|^{2^q}_{L^{2}_x}+\tau\sum_{m=1}^M\|\bfu_{h,m}\|^{2^{q}-2}_{L^{2}_x}\|\nabla\bfu_{h,m}\|^2_{L^2_x}\bigg]&\leq\,c(q,T)\E\big[\|\bfu_{h,0}\|^{2^q}_{L^{2}_x}+1\big]
\end{align}
uniformly in $h$ and $\tau$.
 It is also shown there
that the sequence convergences in law to a martingale solution to \eqref{eq:SNS}. We strengthen this result in short-time (where the stopping times $\mathfrak j_R$ and $\mathfrak m_R$ are introduced below \eqref{eq:jR} and \eqref{eq:mR}, respectively)
by proving an optimal convergence rate with respect to convergence in probability
in the following theorem. Here we suppose that
$R=R(\tau,h)\leq c^{-1/2}\sqrt{-\varepsilon\min\{\log(\tau),\log(h^2)\}}$ as $\tau,h\rightarrow 0$, which inlcudes, in particular, any choice of fixed $R\in\N$.
\begin{theorem}\label{thm:maintx}
Let $\bfu_0\in L^8(\Omega,W^{2,2}_{\Div}(\mt))$ be $\F_0$-measurable and assume that $\Phi$ satisfies \eqref{eq:phi0}--\eqref{eq:phi2b}. Let $$(\bfu,(\mathfrak{t}_R)_{R\in\N},\mathfrak{t})$$ be the unique maximal global strong solution to \eqref{eq:SNS} from Theorem \ref{thm:inc2d}.
Then we have for any $\xi>0$, $\alpha<1$ and $\beta<1$, 
\begin{align*}
&\mathbb P\bigg[\max_{1\leq m\leq \mathfrak m_R\wedge \mathfrak j_R}\|\bfu(t_m)-\bfu_{h,m}\|_{L^2_x}^2+\sum_{m=1}^{\mathfrak m_R\wedge \mathfrak j_R} \tau\|\nabla\bfu(t_m)-\nabla\bfu_{h,m}\|_{L^2_x}^2>\xi\,(\tau^{2\alpha}+h^{2\beta})\bigg]\rightarrow 0
\end{align*}
as $\tau\rightarrow0$,
where $(\bfu_{h,m})_{m=1}^M$ is the solution to \eqref{txdiscr}.
\end{theorem}
\begin{remark}
It is possible to obtain Theorem \ref{thm:maintx} by a direct comparison between the space-time discretisation and the exact solution (avoiding the time discretisation is an intermediate step) as in done \cite[Section 4]{BrPr} in the 2D case. The advantage of such an approach is that the stopping time $\mathfrak j_R$ is not needed. We believe, however, that the plainly temporal error as well as the error between the temporal and spatio-temporal discretisation given in Theorem \ref{thm:mainx} below are of independent interest.
\end{remark}
Theorem \ref{thm:maintx} follows from combining Theorem \ref{thm:maint} with the following result concerning the error between the temporal and spatio-temporal discretisation, the proof of which is the main aim of this section. Here we suppose that
$R=R(h)\leq c^{-1/2}\sqrt{-\varepsilon \log(h^2)}$ as $h\rightarrow 0$.

\begin{theorem}\label{thm:mainx}
Let $\bfu_0\in L^8(\Omega,W^{2,2}_{\Div}(\mt))$ be $\F_0$-measurable and assume that $\Phi$ satisfies \eqref{eq:phi0}--\eqref{eq:phi2b}. Let $(\bfu_{m})_{m=1}^M$ be the solution to \eqref{tdiscr}.
Then we have for any $\xi>0$, $\alpha<\frac{1}{2}$ $\beta<1$, 
\begin{align*}
&\mathbb P\bigg[\max_{1\leq m\leq \mathfrak j_{R}}\|\bfu_m-\bfu_{h,m}\|_{L^2_x}^2+\sum_{m=1}^{\mathfrak j_{R}} \tau\|\nabla\bfu_m-\nabla\bfu_{h,m}\|_{L^2_x}^2>\xi\,h^{2\beta}\bigg]\rightarrow 0
\end{align*}
as $\tau,h\rightarrow0$,
where $(\bfu_{h,m})_{m=1}^M$ is the solution to \eqref{txdiscr}.
\end{theorem}
\begin{proof}
Define the error $\bfe_{h,m}=\bfu_m-\bfu_{h,m}$ for $m\in\{0,1,\dots,\mathfrak j_R\}$. Subtracting \eqref{eq:pressuret} and \eqref{txdiscr} we obtain
\begin{align*}
\begin{aligned}
\int_{\mt}&\bfe_{h,m}\cdot\bfvarphi \dx +\mu\tau\int_{\mt}\Big(\nabla\bfu_m-\nabla\bfu_{h,m}\Big):\nabla\bfphi\dx\\&=\int_{\mt}\bfe_{h,m-1}\cdot\bfvarphi \dx-\tau\int_{\mt}\Big((\nabla\bfu_m)\bfu_{m-1}-\big((\nabla\bfu_{h,m})\bfu_{h,m-1}+(\Div\bfu_{h,m-1})\bfu_{h,m}\big)\Big)\cdot\bfphi\dx\\
&+\int_{\mt}\big(\Phi(\bfu_m)-\Phi(\bfu_{h,m-1})\big)\,\Delta_mW\cdot \bfvarphi\dx\\
&-\int_{\mt}\nabla\Delta^{-1}\Div\Phi(\bfu_{m-1})\,\Delta_mW\cdot \bfvarphi\dx+\tau\int_{\mt}\pi_m^{\mathrm{det}}\,\Div\bfphi\dx
\end{aligned}
\end{align*}
for every $\bfphi\in V_{\Div}^h(\mt)$.
Setting $\bfphi=\Pi_h\bfe_{h,m}$ and applying again the identity $\bfa\cdot(\bfa-\bfb)=\frac{1}{2}\big(|\bfa|^2-|\bfb|^2+|\bfa-\bfb|^2\big)$ for $\bfa,\bfb\in\mathbb R^3$ we gain
\begin{align}\label{eq:0207}
\begin{aligned}
\int_{\mt}&\frac{1}{2}\big(|\Pi_h\bfe_{h,m}|^2-|\Pi_h\bfe_{h,m-1}|^2+|\Pi_h\bfe_{h,m}-\Pi_h\bfe_{h,m-1}|^2\big) \dx+\mu\tau\int_{\mt}|\nabla\bfe_{h,m}|^2\dx\\
&=\mu\tau\int_{\mt}\nabla\bfe_{h,m}:\nabla\big(\bfu_{m}-\Pi_h\bfu_{m}\big)\dx\\
&-\tau\int_{\mt}\Big((\nabla\bfu_m)\bfu_{m-1}-\big((\nabla\bfu_{h,m})\bfu_{h,m-1}+(\Div\bfu_{h,m-1})\bfu_{h,m}\big)\Big)\cdot\Pi_h\bfe_{h,m}\dx\\
&+\tau\int_{\mt}\pi_m^{\mathrm{det}}\,\Div\Pi_h\bfe_{h,m}\dx\\
&+\int_{\mt}\big(\Phi(\bfu_m)-\Phi(\bfu_{h,m-1})\big)\,\Delta_mW\cdot \Pi_h\bfe_{h,m}\dx\\
&-\int_{\mt}\nabla\Delta^{-1}\Div\Phi(\bfu_{m-1})\,\Delta_mW\cdot \Pi_h\bfe_{h,m}\dx\\
&=I_1(m)+\dots +I_5(m).
\end{aligned}
\end{align}
Now we take the maximum with respect to $m\in\{1,\dots,\mathfrak j_R\}$. The terms $I_1(m)$ and $I_3(m)$ can be estimated as in \cite[Section 4]{BrDo} and \cite[Section 4]{BrPr} leading to
\begin{align*}
I_1(m)
&\leq \,\delta\,\tau\int_{\mt}|\nabla\bfe_{h,m}|^2\dx+c(\delta)\,\tau h^2 \int_{\mt}|\nabla^2 \bfu_m|^2\dx,\\
I_3(m)
&\leq \,\delta\tau\,\int_{\mt}|\nabla\bfe_{h,m}|^2\dx+c(\delta)\tau h^2\|\bfu_m\|_{W^{2,2}_x}^2+c(\delta)\tau h^2\,\int_{\mt}|\nabla\pi_m^{\mathrm{det}}|^2\dx,
\end{align*}
for $ \delta>0$ arbitrary.
The $\delta$-terms can be absorbed, while the (sum from $m=1,\dots,\mathfrak j_R$ of the) expectations of the other two terms can be bounded by $h^2 e^{cR^2}$ using Corollary \ref{cor:3.1} (a) and Lemma \ref{lemma:3.2}. More care is required for $I_2(m)$. We argue in the spirit of of \cite[bounds for $I_4(m)$ in the proof of Thm. 4.2]{BrPr} but working with 3D embeddings and the definition of $\mathfrak j_R$.
First we write
\begin{align*}
I_2(m)&=I_2^1(m)+I_2^2(m)+I_2^3(m),\\
I_2^1(m)&=-\tau\int_{\mt}(\bfu_{m-1}\cdot\nabla)\bfe_{h,m}\cdot\big(\bfu_m-\Pi_h\bfu_m\big)\dx,\\
I_2^2(m)&=\tau\int_{\mt}(\bfe_{h,m-1}\cdot\nabla)\bfe_{h,m}\cdot\big(\bfu_m-\Pi_h\bfu_m\big)\dx\\
&+\tau\int_{\mt}(\Div\bfe_{h,m-1})\bfe_{h,m}\cdot\big(\bfu_m-\Pi_h\bfu_m\big)\dx,\\
I_2^3(m)&=-\tau\int_{\mt}(\bfe_{h,m-1}\cdot\nabla)\Pi_h\bfe_{h,m}\cdot\bfu_m\dx\\
&-\tau\int_{\mt}(\Div\bfe_{h,m-1})\Pi_h\bfe_{h,m}\cdot\bfu_m\dx.
\end{align*}
We obtain for any $\delta>0$ and $m\leq\mathfrak j_R$
\begin{align*}
I_2^1(m)&\leq\,\tau\|\nabla\bfe_{h,m}\|_{L^2_x}\|\bfu_{m-1}\|_{L^\infty_x}\|\bfu_m-\Pi_h\bfu_m\|_{L^2_x}\\
&\leq\,\tau Rh^2\|\nabla\bfe_{h,m}\|_{L^2_x}\|\nabla^2\bfu_m\|_{L^2_x}\\
&\leq \,\delta\tau\|\nabla\bfe_{h,m}\|^2_{L^2_x}+c(\delta) h^4 R^2\tau\|\nabla^2\bfu_m\|_{L^2_x}^2
\end{align*}
by the embedding $W^{2,2}(\mt)\hookrightarrow L^\infty(\mt)$ and the approximability of $\Pi_h$ from \eqref{eq:stab'}. The first term can be absorbed for $\kappa$ small enough, whereas the second one (in summed form and expectation) is bounded by $h^4e^{cR^2}$ due Corollary \ref{cor:3.1} (a). Similarly, we have
\begin{align*}
I_2^2(m)&\leq \tau\|\nabla\bfe_{h,m}\|_{L^2_x}\|\bfe_{h,m-1}\|_{L^3_x}\|\bfu_m-\Pi_h\bfu_m\|_{L^6_x}\\
&+\tau\|\nabla\bfe_{h,m-1}\|_{L^2_x}\|\bfe_{h,m}\|_{L^3_x}\|\bfu_m-\Pi_h\bfu_m\|_{L^6_x}\\
&\leq \tau\|\nabla\bfe_{h,m}\|_{L^2_x}\|\bfe_{h,m-1}\|_{L^2_x}^{\frac{1}{2}}\|\nabla\bfe_{h,m-1}\|^{\frac{1}{2}}_{L^2_x}\|\bfu_m-\Pi_h\bfu_m\|_{W^{1,2}_x}\\
&+\tau\|\nabla\bfe_{h,m-1}\|_{L^2_x}\|\bfe_{h,m}\|_{L^2_x}^{\frac{1}{2}}\|\nabla\bfe_{h,m}\|_{L^2_x}^{\frac{1}{2}}\|\bfu_m-\Pi_h\bfu_m\|_{W^{1,2}_x}\\
&\leq\,\delta\tau\Big(\|\nabla\bfe_{h,m-1}\|^2_{L^2_x}+\|\nabla\bfe_{h,m}\|^2_{L^2_x}\Big)\\
&+c(\delta)\,\tau \,h^{4}\Big(\max_{1\leq n\leq m}\|\bfe_{h,n}\|_{L^2_x}^2\Big)\|\nabla^2\bfu_m\|^4_{L^2_x}.
\end{align*}
The last term (in summed form, for $m=1,\dots,\mathfrak j_R,$ and expectation) can be controlled by Lemma \ref{cor:3.1} (c) (with $q=3$) and \ref{lem:4.1} (with $q=2$).
Finally, by definition of $\mathfrak j_R$,
\begin{align*}
I_2^3(m)&\leq \tau\|\nabla\Pi_h\bfe_{h,m}\|_{L^2_x}\|\bfe_{h,m-1}\|_{L^3_x}\|\bfu_m\|_{L^6_x}\\
&+ \tau\|\nabla\bfe_{h,m-1}\|_{L^2_x}\|\Pi_h\bfe_{h,m}\|_{L^3_x}\|\bfu_m\|_{L^6_x}\\
&\leq \tau\|\nabla\bfe_{h,m}\|_{L^2_x}\|\bfe_{h,m-1}\|_{L^2_x}^{\frac{1}{2}}\|\nabla\bfe_{h,m-1}\|^{\frac{1}{2}}_{L^2_x}\|\bfu_m\|_{W^{1,2}_x}\\
&+ \tau\|\nabla\bfe_{h,m-1}\|_{L^2_x}\|\Pi_h\bfe_{h,m}\|_{L^2_x}^{\frac{1}{2}}\|\nabla\Pi_h\bfe_{h,m}\|^{\frac{1}{2}}_{L^2_x}\|\bfu_m\|_{W^{1,2}_x}\\
&\leq\,\delta\Big(\|\nabla\bfe_{h,m-1}\|^2_{L^2_x}+\|\nabla\bfe_{h,m}\|^2_{L^2_x}\Big)+c(\delta)\,\tau\,R^4\Big(\max_{1\leq n\leq m}\|\bfe_{h,n}\|^2_{L^2_x}\Big)\\
&+c(\delta)\,\tau\|\nabla(\bfu_m-\bfu_{m-1})\|^2_{L^2_x}\Big(\max_{1\leq n\leq m}\|\nabla\bfu_{n}\|_{L^2_x}^2\Big)\Big(\max_{0\leq n\leq m}\|\bfe_{h,n}\|_{L^2_x}^2\Big)\\
&+ c(\delta)\,\tau\|\nabla(\bfu_m)-\Pi_h\bfu_m)\|^2_{L^{2}_x}+c(\delta)\,\tau R^4\|\bfu_m-\Pi_h\bfu_m\|^2_{L^2_x}.
\end{align*}
for any $m\leq \mathfrak j_R$.
The last term in the second line will be dealt with by Gronwall's lemma leading to a constant of the form $c e^{cR^4}$. The last term in the second line (in summed form, for $m=1,\dots,\mathfrak j_R,$ and expectation) can be controlled by \eqref{eq:1007}, Lemma \ref{cor:3.1} (c) and \ref{lem:4.1} (each of them with with $q=3$).
The final line is bounded by 
$c(\delta)\,\tau R^4h^{2}\|\bfu_m\|^2_{W^{2,2}_x}$ using \eqref{eq:stab'} and hence can be controlled by Lemma \ref{lem:reg} (c).

In order to estimate the stochastic term $I_4(m)$ we write
\begin{align*}
\mathscr N_{m,1}
:=\sum_{n=1}^m I_4(m)
&=\sum_{n=1}^m \int_{t_{n-1}}^{t_{n}}\int_{\mt}\big(\Phi(\bfu_{n-1})-\Phi(\bfu_{h,n-1})\big)\,\dd W\cdot \bfe_{h,n-1}\dx\\
&+ \sum_{n=1}^m\int_{\mt}\int_{t_{n-1}}^{t_n}\big(\Phi(\bfu_{n-1})-\Phi(\bfu_{h,n-1})\big)\,\dd W\cdot (\bfe_{h,n}-\bfe_{h,n-1})\dx\\
&=:\mathscr N_{m,1}^1+\mathscr N_{m,1}^2.
\end{align*}
Using that $\mathfrak j_R$ is an $(\mathfrak F_{t_m})$-stopping time, we can argue as in \cite[Section 4]{BrDo} and \cite[Section 4]{BrPr} obtaining
\begin{align*}
&\E\bigg[\max_{1\leq m\leq M}\big|\mathscr N_{m\wedge \mathfrak j_R,1}^1\big|\bigg]\\
&\leq\,\delta\,\E\bigg[\max_{0\leq m\leq \mathfrak j_R}\|\Pi_h\bfe_{h,m}\|^2_{L^2_x}\bigg]+\,c(\delta)\,\E\bigg[\tau\sum_{m=1}^{\mathfrak j_R}\|\Pi_h\bfe_{h,m-1}\|_{L^2_x}^2\bigg]\\
&+c(\delta)h^2\,\E\bigg[\tau\sum_{n=1}^{\mathfrak j_R}\|\nabla\bfu_{m-1}\|_{L^2_x}^2\bigg]
\end{align*}
as well as
\begin{align*}
&\E\bigg[\max_{1\leq m\leq M}|\mathscr N_{m\wedge \mathfrak j_R,1}^2|\bigg]\\
&\leq \,\delta\,\E\bigg[ \sum_{m=1}^{\mathfrak j_R} \|\Pi_h\bfe_{h,m}-\Pi_h\bfe_{h,m-1}\|_{L^2_x}^2 \bigg]+ +\,c(\delta)\,\E\bigg[\tau\sum_{m=1}^{\mathfrak j_R}\|\Pi_h\bfe_{h,m-1}\|_{L^2_x}^2\bigg]\\
&+c(\delta)h^2\,\E\bigg[\tau\sum_{n=1}^{\mathfrak j_R}\|\nabla\bfu_{m-1}\|_{L^2_x}^2\bigg].
\end{align*}
In both estimates, the first term can be absorbed, the second one can be handled by Gronwall's lemma, and the last one is bounded by $h^2$ using Lemma \ref{lemma:3.1a}.

In order to estimate $I_5(m)$ we write
\begin{align*}
\mathscr N_{m,2}&:=\sum_{n=1}^mI_5(n)=
\sum_{n=1}^m\int_{\mt}\int_{t_{n-1}}^{t_n}
\big(\mathrm{Id}-\Pi_h^\pi\big)\Delta^{-1}\Div\Phi(\bfu_{n-1})\,\dd W\,\Div \Pi_h\bfe_{h,n-1}\dx\\
&+\sum_{n=1}^m\int_{\mt}\int_{t_{n-1}}^{t_n}
\big(\mathrm{Id}-\Pi_h^\pi\big)\Delta^{-1}\Div\Phi(\bfu_{n-1})\,\dd W\,\Div (\Pi_h\bfe_{h,n}-\Pi_h\bfe_{h,n-1})\dx\\
&=:\mathscr N_{m,2}^1+\mathscr N_{m,2}^2.
\end{align*}
Following again \cite{BrDo} we have
\begin{align*}
&\E\bigg[\max_{1\leq m\leq M}\big|\mathscr M_{m\wedge \mathfrak j_R,2}^1\big|\bigg]\\
&\leq\,c(\delta) h^4\,\E\bigg[\max_{1\leq n\leq \mathfrak j_R}\|\nabla\bfu_{n-1}\|^2_{L^2_x}\bigg]+\,\delta\,\E\bigg[\tau\sum_{n=1}^{\mathfrak j_R}\|\nabla\Pi_h\bfe_{h,n}\|_{L^{2}_x}^2\bigg].
\end{align*}
The first term is bounded by
$h^4e^{cR^2}$ using Corollary \ref{cor:3.1} (a) (recall that $\bfu_0\in L^2(\Omega;W^{1,2}(\mt))$). The second term
can be absorbed given the appropriate choice of $\delta$. Finally,
\begin{align*}
\E\bigg[\max_{1\leq m\leq M}|\mathscr N_{m\wedge \mathfrak j_R,2}^2|\bigg]
&\leq  \,\delta \,\E\bigg[ \sum_{n=1}^{\mathfrak j_R} \big\|\Pi_h\bfe_{h,n}-\Pi_h\bfe_{h,n-1}\big\|_{L^{2}_x}^2\bigg]\\ &+c(\delta) h^{2}\,\E\bigg[\tau\sum_{n=1}^{\mathfrak j_R}\big(1+\big\| \bfu_{n-1}\|_{L^2_x}^2\big)\dt\bigg],
\end{align*}
where the first term can be absorbed and the second one is bounded by $h^2$ on account of Lemma \ref{lemma:3.1a}.
We conclude that
\begin{align*}
&\mathbb E\bigg[\max_{1\leq m\leq \mathfrak j_R}\|\bfu_m-\bfu_{h,m}\|_{L^2_x}^2+\sum_{m=1}^{\mathfrak j_R} \tau\|\nabla\bfu_m-\nabla\bfu_{h,m}\|_{L^2_x}^2\bigg]\leq\,ch^2e^{cR^2}.
\end{align*}
Since $R=R(h)\leq c^{-1/2}\sqrt{-\varepsilon \log(h^2)}$, where $\varepsilon>0$ is arbitrary, the claim follows now by applying Markov's inequality.
\end{proof}

\end{document}